\documentclass[reqno,10pt]{amsart}
\usepackage{amsmath,amssymb,amsthm,mathrsfs}
\usepackage{longtable}
\usepackage{enumerate}
\usepackage{graphicx}
\usepackage{subcaption}
\usepackage{wasysym}
\usepackage{upgreek}
\usepackage[dvipsnames]{xcolor}
\usepackage{cancel}
\usepackage{soul}
\usepackage{cite}
\usepackage{mathtools}

\allowdisplaybreaks

\newtheorem{theorem}{Theorem}[section]
\newtheorem{lemma}[theorem]{Lemma}
\newtheorem{prop}[theorem]{Proposition}
\newtheorem{corollary}[theorem]{Corollary}
\newtheorem{definition}[theorem]{Definition}

\newcommand{\un}{\mathbf{u}}

\newcommand{\R}{\mathbb{R}}

\numberwithin{equation}{section}

\usepackage{xcolor}

\title[On classical OP and the Cholesky factorization of Hankel matrices]{On classical orthogonal polynomials and the Cholesky factorization of a class of Hankel matrices}

\begin{document}

\author[M. E. Marriaga, G. Vera de Salas, M. Latorre, R. Muñoz Alcázar]
{Misael E. Marriaga*, Guillermo Vera de Salas, Marta Latorre, Rubén Muñoz Alcázar}

\address[M. E. Marriaga]{Departamento de Matem\'atica Aplicada, Ciencia e Ingenier\'ia de Materiales y
Tecnolog\'ia Electr\'onica, Universidad Rey Juan Carlos (Spain)}
\email{misael.marriaga@urjc.es} 
\address[G. Vera de Salas]{Departamento de Matem\'atica Aplicada, Ciencia e Ingenier\'ia de Materiales y
Tecnolog\'ia Electr\'onica, Universidad Rey Juan Carlos (Spain)}
\email{guillermo.vera@urjc.es} 
\address[M. Latorre]{Departamento de Matem\'atica Aplicada, Ciencia e Ingenier\'ia de Materiales y
Tecnolog\'ia Electr\'onica, Universidad Rey Juan Carlos (Spain)}
\email{marta.latorre@urjc.es}
\address[R. Muñoz Alc\'azar]{Departamento de \'Algebra, Geometr\'ia y Topolog\'ia, Universidad de M\'alaga (Spain)}
\email{rubenm.alcazar@uma.es}

\thanks{*Corresponding Author: Misael E. Marriaga}

\date{\today}

\begin{abstract}
Classical moment functionals (Hermite, Laguerre, Jacobi, Bessel) can be characterized as those linear functionals whose moments satisfy a second order linear recurrence relation. In this work, we use this characterization to link the theory of classical orthogonal polynomials and the study of Hankel matrices whose entries satisfy a second order linear recurrence relation. Using the recurrent character of the entries of such Hankel matrices, we give several characterizations of the triangular and diagonal matrices involved in their Cholesky factorization and connect them with a corresponding characterization of classical orthogonal polynomials.
\end{abstract}

\subjclass[2010]{15A23, 15A23, 15A63, 33C45, 33D45}

\keywords{Orthogonal polynomials, Hankel matrices, Cholesky factorization.}

\maketitle

\section{Introduction}

Classical orthogonal polynomials (Hermite, Laguerre, Jacobi, and Bessel) have been characterized using different approaches. For instance, they can be characterized in terms of differential equations  \cite{B29}, their derivatives (\cite{Hahn35,Krall36}), structure relations (\cite{ASC72,Ger40, MBP94}), and a Rodrigues formula (\cite{Tri70}), among others (see \cite{GMM21,MP94} and the references therein). More recently, an approach that uses linear functionals and duality was introduced by P. Maroni (\cite{Ma87}). In all of these approaches, the starting point is to use the basis of monomials to represent polynomials and state the results.

However, an interesting (and most recent) approach is to start from the theory of semi-infinite matrices. The bibliography on this subject has grown greatly in the last years, and it has become increasingly difficult to do a comprehensive review of all the references. Hence, we refer the reader to \cite{V13,V2021} (and the refereces therein) where the algebra of infinite triangular matrices and the algebra of infinite Hessenberg matrices are used to study some aspects of orthogonal polynomials, and to \cite{{D20},M2021} (and the references therein) where the main tool is the Cholesky factorization of Gram matrices of bilinear forms. We remark that the Cholesky factorization proves to be quite fruitful in the study of non standard orthogonality such as multiple, matrix, Sobolev, and multivariate orthogonality as well as orthogonality on the unit circle of the complex plane, and have successfully found its way into applications in random matrices, Toda lattices, integrable systems, Riemann-Hilbert problems, Painlevé equations, and Darboux transformations, among others topics.

Our goal is to contribute to the link between matrix factorization and orthogonal polynomials. In particular, we deal with several characterizations of classical orthogonal polynomials. However, we shift our paradigm from infinite matrices to the \textit{finite} Gram matrix $G_n$ associated with a bilinear form defined on the linear space of polynomials of degree at most $n\geqslant 0$. For standard orthogonality, $G_n$ is a Hankel matrix (all of its antidiagonals are constant). Taking into account that the moments of a linear functional associated with a family of classical orthogonal polynomials satisfy a second order linear recurrence relation (\cite{MBP94}), we can say that this paper deals with Hankel matrices with an additional structure: the entries of $G_n$ satisfy such recurrence relation. In this way, we can extend the bilinear form to the linear space of polynomials of degree at most $n+1$ by constructing a new Gram matrix $G_{n+1}$ by means of bordering $G_n$ with a new row and column whose entries are obtained using the recurrence relation and the entries of $G_n$. The resulting matrix $G_{n+1}$ will also be a Hankel matrix with the additional structure mentioned above. Consequently, it will be possible to prove by induction that the properties satisfied by $G_n$ are also satisfied by $G_{n+1}$.

The change from infinite matrices to subsequently bordering finite matrices is motivated by an alternative proof of a classical result about the interlacing of zeros of orthogonal polynomials of consecutive degrees found in \cite{GMM21}. This classical result states that if the zeros of any two polynomials of consecutive degrees interlace, then these polynomials are elements of a sequence of orthogonal polynomials associated with a positive definite moment functional. This result can be proved using the Euclidean division algorithm for polynomials. However, this result can also be deduced using the following theorem about interlacing eigenvalues of Hermitian matrices found in \cite[p. 185]{HJ85}:
\begin{theorem}
    Let $n$ be a given positive integer, and let  $\{x_{n,k}\}_{k=1}^n$ and $\{x_{n-1,k}\}_{k=1}^{n-1},$ be two given sequences of numbers such that  	
	$$x_{n,1}<x_{n-1,1}<\cdots <x_{n,k}<x_{n-1,k}<x_{n,k+1}<\cdots< x_{n-1,n-1}<x_{n,n}.$$
Let $\Lambda=\operatorname{diag}(x_{n-1,1},x_{n-1,2},\cdots,x_{n-1,n-1})$. Then there exists  real number $b$ and a real vector $\bar{y}=(y_1,\ldots,y_{n-1})^{\top}\in\R^{n-1}$ such that $(x_{n,k})_{k=1}^n$ is the set of eigenvalues of the real symmetric matrix
$$B=\left(\begin{array}{c|c}
\Lambda & \bar{y}\\
\hline
\\[-6pt]
\bar{y}^{\top}&b
\end{array}\right).$$
\end{theorem}
\noindent From this, it seems reasonable to think that the procedure of bordering matrices encoding information about polynomial sequences is well suited for presenting and deducing results about orthogonal polynomials most likely due to the fact that for all $n\geqslant 0$, the linear space of polynomials of degree at most $n$ is a subspace of the space of polynomials of degree at most $n+1$. In this way, the Gram matrices of bilinear forms associated with classical orthogonal polynomials posses the adequate structure to start exploring our proposed paradigm.

The paper is organized as follows. Section \ref{classicalops} presents basic background on classical orthogonal polynomials and their associated linear functionals, and we introduce classical sequences of real numbers in Section \ref{classicalsequences}. In Section \ref{orthogonality}, we discuss the Cholesky factorization of Hankel matrices obtained from given sequences of real number and its relation to orthogonal polynomials. We present several characterizations of classical sequences of real numbers in Section \ref{characterizations}.

\section{Orthogonal polynomials and linear functionals}\label{classicalops}

For $n\geqslant 0$, let $\Pi_n$ be the linear space of polynomials of degree at most $n$ of a real variable and real coefficients, and let $\Pi=\bigcup_{n\geqslant 0}\Pi_n$.

Let $\Pi^*$ denote the algebraic dual space of $\Pi$. That is, $\Pi^*$ is the linear space of linear functionals defined on $\Pi$,
$$
\Pi^*=\left\{\mathbf{u}:\Pi \rightarrow \mathbb{R} \ \ : \ \ \mathbf{u} \text{ is linear} \right\}.
$$
We denote by $\langle \mathbf{u}, p\rangle$ the image of the polynomials $p$ under the linear functional $\mathbf{u}$.

Any linear functional $\mathbf{u}$ is completely defined by the values
$$
\mu_n\, :=\,\langle \mathbf{u}, x^n\rangle,\quad n\geqslant 0,
$$
and extended by linearity to all polynomials, where $\mu_n$ is called the $n$-th moment of $\mathbf{u}$. Therefore, we refer to $\mathbf{u}$ as a moment functional.

A moment functional $\mathbf{u}$ is called positive definite if $\langle \mathbf{u}, p^2\rangle >0$ for every non zero polynomial $p\in \Pi$.

Let $\mathbf{u}$ be a moment functional. A sequence of polynomials $\{P_n(x)\}_{n\geqslant 0}$ is called an orthogonal polynomial sequence (OPS) with respect to $\mathbf{u}$ if 
\begin{enumerate}
    \item[(1)] $\deg\,P_n=n$,
    \item[(2)] $\langle \mathbf{u}, P_n\,P_m\rangle = h_n\,\delta_{n,m}$, with $h_n\ne 0$.
\end{enumerate}
Here $\delta_{n,m}$ denotes the Kronecker delta defined as
$$
\delta_{n,m}=\left\{\begin{array}{ll}
1, & n=m,\\
0, & n\ne m.
\end{array}
\right.
$$
If there exists an OPS associated with $\mathbf{u}$, then $\mathbf{u}$ is called quasi-definite. Positive definite moment functionals are quasi-definite.

Observe that an OPS $\{P_n(x)\}_{n\geqslant 0}$ constitutes a basis for $\Pi$. If for all $n\geqslant 0$, the leading coefficient of $P_n(x)$ is 1, then $\{P_n(x)\}_{n\geqslant 0}$ is called a monic orthogonal polynomial sequence (MOPS).

Given a moment functional $\mathbf{u}$ and a polynomial $q(x)$, we define the left multiplication of $\mathbf{u}$ by $q(x)$ as the moment functional $q\,\mathbf{u}$ such that
$$
\langle q\,\mathbf{u}, p\rangle \,=\,\langle \mathbf{u}, q\,p\rangle, \quad \forall p\in \Pi,
$$
and we define the distributional derivative $D\mathbf{u}$ by
$$
\langle D\mathbf{u}, p\rangle \,=\, -\langle \mathbf{u}, p'\rangle, \quad \forall p\in \Pi.
$$
Moreover, the product rule is satisfied, that is,
$$
D(q\,\mathbf{u})\,=\,q'\,\mathbf{u}+q\,D\mathbf{u}.
$$

\begin{definition}
Let $\mathbf{u}$ be a quasi-definite moment functional, and let $\{P_n(x)\}_{n\geqslant 0}$ be an OPS with respect to $\mathbf{u}$. Then $\mathbf{u}$ is classical if there are nonzero polynomials $\phi(x)$ and $\psi(x)$ with $\deg \phi \leqslant 2$ and $\deg \psi =1$, such that $\mathbf{u}$ satisfies the distributional Pearson equation
\begin{equation}\label{pearson}
    D(\phi\,\mathbf{u})\,=\,\psi\,\mathbf{u}.
\end{equation}
The sequence $\{P_n(x)\}_{n\geqslant 0}$ is called a classical OPS.
\end{definition}

The following characterizations of classical moment functionals and OPS will be of central importance in the sequel.

\begin{theorem}\label{th:classical-char}
Let $\mathbf{u}$ be a quasi-definite moment functional, and $\{P_n(x)\}_{n\geqslant 0}$ its associated MOPS. The following statements are equivalent: 
\begin{enumerate}
    \item[1.] $\mathbf{u}$ is a classical moment functional.
    \item[2.] (Bochner, \cite{B29}) There are nonzero polynomials $\phi(x)$ and $\psi(x)$ with $\deg \phi\leqslant 2$ and $\deg \psi=1$ such that, for $n\geqslant 0$, $P_n(x)$ satisfies
    \begin{equation}\label{bochner-diffeq}
    \phi(x)\,P_n''(x)+\psi(x)\,P_n'(x)\,=\,\lambda_n\,P_n(x),
    \end{equation}
    where $\lambda_n=n\,(\frac{n-1}{2}\phi''+\psi')$.
    \item[3.] (Hahn,\cite{Hahn35}) There is a nonzero polynomial $\phi(x)$ with $\deg \phi\leqslant 2$, such that $\left\{\frac{P_{n+1}'(x)}{n+1} \right\}_{n\geqslant 0}$ is the MOPS associated with the moment functional $\mathbf{v}=\phi(x)\,\mathbf{u}$.
   \item[4.] (First structure relation, \cite{ASC72}) There is a nonzero polynomial with $\deg \phi\leqslant 2$, and real numbers $a_n$, $b_n$, $c_n$, $n\geqslant 1$, with $c_n\ne 0$, such that
   $$
  \phi(x)\,P'_n(x)\,=\,a_n\,P_{n+1}(x)+b_n\,P_{n}(x)+c_n\,P_{n-1}(x), \quad n\geqslant 1.
 $$
    \item[5.] (Second structure relation, \cite{Ger40, MBP94}) There are real numbers $\alpha_n$ and $\beta_n$, $n\geqslant 2$, such that
    \begin{equation}\label{structrel}
    P_n(x)\,=\,\frac{P'_{n+1}(x)}{n+1}+\alpha_n\,\frac{P_n'(x)}{n}+\beta_n\,\frac{P'_{n-1}(x)}{n-1},\quad n\geqslant 2.
    \end{equation}
    \item[6.] (Rodrigues formula, \cite{Tri70}) There is a non zero polynomial $\phi(x)$ with $\deg \phi\leqslant 2$, and a non zero real number $k_n\ne 0$ such that
    $$
    D^n(\phi^n(x)\,\mathbf{u})\,=\,k_n\,P_n(x)\,\mathbf{u}, \quad n\geqslant 0.
    $$
\end{enumerate}
\end{theorem}

It is well known (see \cite{B29} as well as \cite{Krall41}) that, up to affine transformations of the independent variable, the only families of positive definite classical orthogonal polynomials are the Hermite, Laguerre, and Jacobi polynomials. The corresponding moment functionals admit an integral representation of the form
$$
\langle \mathbf{u}, p\rangle \,=\,\int_{I}p(x)\,w(x)\,dx, \quad p\in \Pi,
$$
where $I=\mathbb{R}$ and $w(x)=e^{-x^2}$ in the Hermite case, $I=(0,+\infty)$ and $w(x)=x^{\alpha}e^{-x}$ with $\alpha>-1$ in the Laguerre case, and $I=(-1,1)$ and $w(x)=(1-x)^{\alpha}(1+x)^{\beta}$ with $\alpha,\beta>-1$ in the Jacobi case. We note that in each case, $w(x)>0$ on $I$ and, thus, we say that $w(x)$ is a weight function.

The definition of classical moment functionals in terms of the distributional Pearson equation not only encompasses positive definite moment functionals associated with weight functions, but includes the non positive case as well. Considering the non positive definite case gives rise to the Bessel classical moment functional satisfying the distributional Pearson equation \eqref{pearson} with $\phi(x)=x^2$ and $\psi(x)=ax+2$. The Bessel functional is quasi-definite when $a\ne -1,-2,\ldots$ Moreover, it has the following integral representation
$$
\langle \mathbf{u}, p\rangle = \int_c p(z)\,w(z)\,dz, \quad p\in \Pi,
$$
where $w(z)=(2\,\pi\,i)^{-1}z^{a-2}e^{-2/z}$, and $c$ is the unit circle oriented in the counter-clockwise direction. 

Observe that from Theorem \ref{th:classical-char}, if $\mathbf{u}$ is a classical moment functional satisfying~\eqref{pearson}, then $\mathbf{v}=\phi(x)\,\mathbf{u}$ is a classical moment functional satisfying the Pearson equation
$$
D(\phi\,\mathbf{v})\,=\,(\psi+\phi')\,\mathbf{v}.
$$
Iterating this idea, we get that the high-order derivatives of classical orthogonal polynomials are again classical orthogonal polynomials of the same type.

\begin{theorem}[\cite{Hahn35,Krall41, Krall36}]\label{th:polynomials-Q}
Let $\un$ be a classical moment functional satisfying \eqref{pearson}, and $\{P_n(x)\}_{n\geqslant 0}$ its corresponding MOPS. For $k\geqslant 0$, let $\mathbf{v}_k=\phi^k(x)\,\mathbf{u}$ and $\{Q_{n,k}(x)\}_{n\geqslant 0}$ be the sequence of polynomials given by 
\begin{equation}\label{eq-polynomials-Q}
Q_{n,k}(x):=\dfrac{1}{(n+1)_k}P_{n+k}^{(k)}(x),  \quad n\geqslant 0,   
\end{equation}
where $p^{(k)}$ is the $k$-th derivative of $p$, and $(\nu)_k=\nu\,(\nu+1)\cdots (\nu+k-1)$, $(\nu)_0=1$, denotes the Pochhammer symbol. Then, for each $k\geqslant 0$, $\{Q_{n,k}(x)\}_{n\geqslant 0}$ is  a MOPS associated with the moment functional $\mathbf{v}_k$, satisfying
$$
D(\phi\,\mathbf{v}_k)\,=\,\psi_k\,\mathbf{v}_k,
$$
where $\psi_k(x)=\psi(x)+k\,\phi'(x)$. Hence, $\mathbf{v}_k$ is a classical moment functional.
\end{theorem}

\section{Classical sequences of numbers}\label{classicalsequences}

This section is devoted to presenting the definition of classical moment functionals from a different approach. We start by introducing sequences of real numbers that satisfy a second order recurrence relation and use them to construct linear functionals defined on $\Pi$.

\begin{definition}\label{def:classical sequence}
Let $\{\mu_n\}_{n\geqslant 0}$ be a sequence of real numbers with $\mu_0\ne 0$. Then $\{\mu_n\}_{n\geqslant 0}$ is a pre-classical sequence if there are real numbers $a,b,c,d,e$ satisfying 
$$
|a|+|b|+|c|>0, \quad n\,a+d\ne 0 \quad n\geqslant 0,
$$
such that the following holds
\begin{equation}\label{eq:ttr-moments}
(n\,a+d)\,\mu_{n+1}+(n\,b+e)\,\mu_n+n\,c\,\mu_{n-1}=0, \quad n \geqslant 0.
\end{equation}
By convention, $\mu_n=0$ whenever $n<0$.
\end{definition}

Let $\{\mu_n\}_{n\geqslant 0}$ be a pre-classical sequence of real numbers. Then it is possible to define a functional $\mathbf{u}$ as
$$
\mu_n\, :=\,\langle \mathbf{u}, x^n\rangle,\quad n\geqslant 0,
$$
and extend it by linearity to all polynomials, where $\mu_n$ is called the $n$-th moment of $\mathbf{u}$. Therefore, we refer to $\mathbf{u}$ as a pre-classical moment functional. Observe that the condition $n\,a+d\ne 0$, $n\geqslant 0$, guarantees that $\mathbf{u}$ is completely defined since each moment  
$$
\mu_{n+1}=-\frac{1}{n\,a+d}[(n\,b+e)\,\mu_n+n\,c\,\mu_{n-1}], \quad n\geqslant 0,
$$
is well-defined.

The recurrence relation \eqref{eq:ttr-moments} can be passed down to the pre-classical moment functional associated with $\{\mu_n\}_{n\geqslant 0}$.

\begin{theorem} \label{th:pearson}
A sequence $\{\mu_n\}_{n\geqslant 0}$ is pre-classical if and only if there are non zero polynomials $\phi(x)$ and $\psi(x)$ with $\deg \phi\leqslant 2$, $\deg \psi =1$, and $\frac{n}{2}\phi''+\psi'\ne 0$ for $n\geqslant 0$, such that the moment functional $\mathbf{u}$ defined by $\mu_n=\langle \mathbf{u},x^n\rangle$ satisfies the distributional Pearson equation
\begin{equation*}
    D(\phi\,\mathbf{u})\,=\,\psi\,\mathbf{u}.
\end{equation*}
\end{theorem}

\begin{proof}
Suppose that $\mathbf{u}$ satisfies \eqref{pearson} with non zero $\phi(x)=a\,x^2+b\,x+c$ and $\psi(x)=d\,x+e$ with $\deg \psi=1$ such that $0\ne \frac{n}{2}\phi''+\psi'=n\,a+d$ for $n\geqslant 0$. Then
$$
\langle \mathbf{u}, \phi\,p'+\psi\,p\rangle =0, \quad \forall p\in \Pi.
$$
In particular,
$$
0=\langle \mathbf{u}, n\,\phi\,x^{n-1}+\psi\,x^n\rangle = (n\,a+d)\,\mu_{n+1}+(n\,b+e)\,\mu_n+n\,c\,\mu_{n-1} , \quad n\geqslant 0.
$$
Therefore $\{\mu_n\}_{n\geqslant 0}$ is a pre-classical sequence of real numbers and, thus, $\mathbf{u}$ is a pre-classical moment functional. It is easy to verify that the implications in the opposite direction holds by inverting each of the previous steps.
\end{proof}

For any sequence $\{\mu_n\}_{n\geqslant 0}$, we can define the sequence of matrices $\{G_n\}_{n\geqslant 0}$ where $G_n$ is an $(n+1)\times (n+1)$ matrix given by $G_0=\mu_0$ and 
\begin{equation}\label{def:Gn}
G_n= \left[\begin{array}{ccc|c}
 & & & \mu_n \\
 & G_{n-1} & &\vdots \\
 & & &\mu_{2n-1}\\
\hline 
\mu_n & \ldots & \mu_{2n-1} & \mu_{2n}
\end{array}\right]=\begin{bmatrix}
\mu_0 & \mu_1 & \ldots & \mu_n \\
\mu_1 & \mu_2 & \ldots & \mu_{n+1} \\
\vdots & \vdots & \ddots & \vdots \\
\mu_n & \mu_{n+1} & \dots & \mu_{2n}
\end{bmatrix}, \quad n\geqslant 0.
\end{equation}

In particular, consider the sequence $\{\mu_n\}_{n\geqslant 0}$ with $\mu_0=1$ and $\mu_n=0$ for $n\geqslant 1$. Observe that this sequence corresponds to the linear functional $\boldsymbol{\delta}\in \Pi^*$, known as the Dirac delta, defined as
$$
\langle \boldsymbol{\delta}, p\rangle =p(0), \quad \forall p\in \Pi.
$$
In this case, we have $G_0=1$,
$$
G_n= \begin{bmatrix}
1 & 0 & \ldots & 0 \\
0 & 0 & \ldots & 0 \\
\vdots & \vdots & \ddots & \vdots \\
0 & 0 & \dots & 0
\end{bmatrix},
$$
and $\det G_n =0$ for $n\geqslant 1$. In the sequel, we will need to exclude this and other similar cases and, therefore, we impose that $\det G_n \ne 0$ for $n\geqslant 0$. Hence, we have the following definition.

\begin{definition}\label{def:classical}
A pre-classical sequence $\{\mu_n\}_{n\geqslant 0}$ is classical if the sequence of matrices $\{G_n\}_{n\geqslant 0}$ defined as in \eqref{def:Gn} satisfy $\det G_n \ne 0$ for $n\geqslant 0$. The moment functional defined by $\mu_n=\langle \mathbf{u}, x^n\rangle$ is called a classical moment functional.
\end{definition}

In the sequel, some orthogonal bases for $\mathbb{R}^{n+1}$ associated with a classical sequence of numbers will play an important role in our study of such sequences. Therefore, in the following section we discuss the orthogonal structure of $\mathbb{R}^{n+1}$ induced by general sequences of numbers.

\section{Sequences of numbers and orthogonality}\label{orthogonality}

Given an integer $n\geqslant 0$, a sequence of numbers induces a bilinear form in $\mathbb{R}^{n+1}$ whose Gram matrix is $G_n$ defined in \eqref{def:Gn}. In this section, we explore orthogonal bases of $\mathbb{R}^{n+1}$ associated with a sequence of numbers and use it to construct bases of polynomials in $\Pi_n$ orthogonal with respect to its corresponding moment functional.

For $n\geqslant 0$, we will denote by $\bar{e}_0,\ldots,\bar{e}_{n}$ the columns of the identity matrix $I_{n+1}$, that is, $I_{n+1}=[\bar{e}_0 \ \bar{e}_1 \ \ldots \ \bar{e}_{n}]$. The set of column vectors $\mathcal{E}_{n}=\{\bar{e}_0,\ldots,\bar{e}_{n}\}$ is called the canonical basis for $\mathbb{R}^{n+1}$.

\begin{definition}\label{def:bilinearform}
Let $\{\mu_n\}_{n\geqslant 0}$ be a sequence of numbers. For $n\geqslant 0$, $\mathfrak{B}_n(\cdot,\cdot)$ denotes the bilinear form defined by
$$
\mathfrak{B}_n(\bar{u},\bar{v})\,=\,\bar{u}^{\top}\,G_n\,\bar{v}, \quad \forall \bar{u},\bar{v}\in \mathbb{R}^{n+1},
$$
and is called the bilinear form associated with $\{\mu_n\}_{n\geqslant 0}$ (relative to the canonical basis of $\mathbb{R}^{n+1}$).
\end{definition}

For $n\geqslant 0$, if $\det G_n>0$, then $\mathfrak{B}_n(\cdot,\cdot)$ is an inner product on $\mathbb{R}^{n+1}$. In this case, we can define the norm
$$
\|\bar{u}\|_n= \sqrt{\mathfrak{B}_n(\bar{u},\bar{u})}, \quad \forall \bar{u}\in \mathbb{R}^{n+1}.
$$

Of course, there are many orthogonal bases for $\mathbb{R}^{n+1}$ associated with $\mathfrak{B}_n$. However, we are interested in orthogonal bases obtained from the Cholesky factorization of $G_n$. Recall that if $\det G_n\ne 0$, there is an $(n+1)\times (n+1)$ unit lower triangular matrix (that is, with $1$'s in its main diagonal) $S_n^{-1}$ and an $(n+1)\times (n+1)$ non singular diagonal matrix $H_n$ such that
$$
G_n=S_n^{-1}\,H_n\,S_n^{-\top}, 
$$
where $S_n^{-\top}=(S_n^{-1})^{\top}$ (see Theorem 4.1.3 in \cite{GV13}). Moreover, this matrix factorization is unique. We can immediately observe that if we write the above identity as
$$
S_n\,G_n\,S_n^{\top}\,=\,H_n,
$$
then we have an orthogonality relation for the columns of $S_n^{\top}$ as we show in the following theorem.

\begin{theorem}\label{th:choleskyorth}
Let $G_n=S_n^{-1}\,H_n\,S_n^{-\top}$ be the Cholesky decomposition of $G_n$. Then the columns of $S_n^{\top}$ form an orthogonal basis for $\mathbb{R}^{n+1}$ with respect to the bilinear form $\mathfrak{B}_n$.
\end{theorem}
\begin{proof}
If $\bar{s}_0,\bar{s}_1,\ldots,\bar{s}_{n}$ are the columns of $S_n^{\top}$, then
$$
H_n=S_n\,G_n\,S_n^{\top}=\begin{bmatrix}
\bar{s}_0^{\top} \\[3pt] \bar{s}_1^{\top} \\ \vdots \\ \bar{s}_{n}^{\top}
\end{bmatrix}\,G_n\, \begin{bmatrix} \bar{s}_0 & \bar{s}_1 & \ldots & \bar{s}_{n} \end{bmatrix}.
$$
Since $H_n$ is a diagonal matrix and $G_n$ is the representation of the bilinear form $\mathfrak{B}_n$, it follows that 
$$
\mathfrak{B}_n(\bar{s}_i,\bar{s}_j)=0 \quad \text{for} \quad i\ne j,
$$
and
$$
\mathfrak{B}_n(\bar{s}_i,\bar{s}_i)=h_i, \quad i=0,1,\ldots,n,
$$
where $h_i$ is the $i$-th non zero entry of $H_n$, that is, $H_n=\text{diag}(h_0,h_1,\ldots,h_{n})$
\end{proof}

The following theorem shows that the Cholesky factorization of $G_n$ and $G_{n+1}$ are related. In fact, the Cholesky factorization of $G_{n+1}$ is obtained by bordering $S_n$ and $H_n$ with a new row and a new column. The proof relies heavily on the expressions for $G_{n}^{-1}$ and $\det G_n$ in terms of cofactors. If $(G_n)_{i,j}$ is the $(i,j)$-cofactor of $G_n$ with $i,j=0,1,\ldots,n$, then
$$
\det G_n = \sum_{k=0}^n\,\mu_{i+k}\,(G_n)_{i,k}, \quad G_n^{-1}=\frac{1}{\det G_n}\begin{bmatrix} (G_n)_{0,0} & (G_{n})_{1,0} & \ldots & (G_n)_{n,0} \\
(G_n)_{0,1} & (G_n)_{1,1} & \ldots & (G_n)_{n,1} \\
\vdots & \vdots & \ddots & \vdots \\
(G_n)_{0,n} & (G_n)_{1,n} & \ldots & (G_n)_{n,n}
\end{bmatrix}.
$$

Hereon, $\mathtt{0}$ will denote the zero matrix of appropriate size.

\begin{theorem}\label{th:choleskySn+1}
Let $n\geqslant 0$, and let $S_n^{-1}\,H_n\,S_n^{-\top}$ be the Cholesky factorization of $G_n$. Then, $G_{n+1}=S_{n+1}^{-1}\,H_{n+1}\,S_{n+1}^{-\top}$ with
$$
S_{n+1}^{\top}=\left[\begin{array}{c|c}S_n^{\top} & \hat{s}_{n+1} \\
\hline
\mathtt{0} & 1 \end{array} \right], \quad H_{n+1}=\left[\begin{array}{c|c}H_n & \mathtt{0}\\
\hline
\mathtt{0} & h_{n+1} \end{array} \right],
$$
where $[\hat{s}_{n+1}^{\top} \ 1 ]^{\top}$ is a vector given by the formal determinant
\begin{equation}\label{eq:choleskySn+1}
\begin{bmatrix}
\hat{s}_{n+1} \\
1
\end{bmatrix}= \frac{1}{\det G_n} \det\left[\begin{array}{ccc|c} & & & \mu_{n+1}\\
 & G_n & & \vdots \\
 & & & \mu_{2n+1} \\
 \hline 
\bar{e}_0 & \ldots &  \bar{e}_n & \bar{e}_{n+1}\end{array} \right],
\end{equation}
and
\begin{equation}\label{eq:hn+1}
h_{n+1}=\frac{\det G_{n+1}}{\det G_n}.
\end{equation}
\end{theorem}

\begin{proof}
Observe that the columns of $[S_n^{\top} \ \mathtt{0}]^{\top}$ are linear combinations of the vectors $\bar{e}_0,\ldots,\bar{e}_n\in \mathbb{R}^{n+2}$ (that is, the first $n+1$ columns of $I_{n+2}$). Then $H_{n+1}=S_{n+1}\,G_{n+1}\,S_{n+1}^{\top}$ with $S_{n+1}^{\top}$ and $H_{n+1}$ as in \eqref{eq:choleskySn+1} and \eqref{eq:hn+1} if and only if
$$
\mathfrak{B}_{n+1}([\hat{s}_{n+1}^{\top} \ 1 ]^{\top},\bar{e}_i)=0, \quad i=0,1,\ldots,n,
$$
and 
$$
\mathfrak{B}_{n+1}([\hat{s}_{n+1}^{\top} \ 1 ]^{\top},[\hat{s}_{n+1}^{\top} \ 1 ]^{\top})=\frac{\det G_{n+1}}{\det G_n}.
$$
The above conditions can be written as a system of linear equations:
\begin{equation*}
G_n\,\hat{s}_{n+1}=-\begin{bmatrix}
\mu_{n+1} \\ \vdots \\ \mu_{2n+1}
\end{bmatrix}.
\end{equation*}
By Cramer's rule, 
\begin{equation}\label{eq:vectorshat}
\hat{s}_{n+1}=\frac{1}{\det G_n}\bigg((G_{n+1})_{n+1,0}\,\bar{e}_0+\cdots+(G_{n+1})_{n+1,n}\,\bar{e}_n \bigg),
\end{equation}
which is the expansion of the determinant \eqref{eq:choleskySn+1} across the last row.

 Moreover, we also have
\begin{equation}\label{eq:choleskyvector}
\hat{s}_{n+1}=-G_n^{-1}\begin{bmatrix}
\mu_{n+1} \\ \vdots \\ \mu_{2n+1}
\end{bmatrix},
\end{equation}
and
\begin{align*}
S_{n+1}\,G_{n+1}\,S_{n+1}^{\top}&= \left[\begin{array}{c|c}S_n & \mathtt{0}  \\
\hline
\\[-8pt]
\hat{s}_{n+1}^{\top} & 1 \end{array} \right]\,\left[\begin{array}{ccc|c}
 & & & \mu_{n+1} \\
 & G_{n} & &\vdots \\
 & & &\mu_{2n+1}\\
\hline 
\mu_{n+1} & \ldots & \mu_{2n+1} & \mu_{2n+2}
\end{array}\right]\,\left[\begin{array}{c|c}S_n^{\top} & \hat{s}_{n+1} \\
\hline
\mathtt{0} & 1 \end{array} \right]\\[5pt]
&= \left[\begin{array}{c|c} H_n & S_n\,B_n \\ \hline \\[-8pt]   C_n\,S_n^{\top} & h_{n+1}  \end{array}\right],
\end{align*}
where 
$$
B_n= G_n\,\hat{s}_{n+1}+\begin{bmatrix}
\mu_{n+1} \\ \vdots \\ \mu_{2n+1}
\end{bmatrix}, \quad C_n= \hat{s}_{n+1}^{\top}\,G_n+\begin{bmatrix}
\mu_{n+1} \\ \vdots \\ \mu_{2n+1}
\end{bmatrix}^{\top}, \quad h_{n+1} = [\hat{s}_{n+1}^{\top} \,|\, 1]\begin{bmatrix}
\mu_{n+1} \\ \vdots \\ \mu_{2n+1} \\ \hline \mu_{2n+2}
\end{bmatrix}.
$$
It follows from \eqref{eq:choleskyvector} that $B_n=\mathtt{0}$ and $C_n=\mathtt{0}$. Finally, using \eqref{eq:vectorshat} and the fact that $(G_{n+1})_{n+1,n+1}=\det G_n$, we obtain
$$
h_{n+1}=\frac{1}{\det G_n}\sum_{j=0}^{n+1} \mu_{n+1+j}(G_{n+1})_{n+1,j}= \frac{\det G_{n+1}}{\det G_n},
$$
which proves \eqref{eq:hn+1}.
\end{proof}

We can reformulate the above discussion in the context of $\Pi_n$ as follows. Given a sequence of numbers $\{\mu_n\}_{n\geqslant 0}$,  denote by $\mathbf{u}$ the moment functional defined by $\mu_n=\langle \mathbf{u}, x^n\rangle$. If $\det G_n\ne 0$, then 
$$
\langle \mathbf{u}, p\,q\rangle, \quad p,q\in \Pi_n,
$$
is a (non degenerate) bilinear form defined on $\Pi_n$. It is easy to see that its Gram matrix relative to the basis of monomials is $G_n$. Let $G_n=S_n^{-1}\,H_n\,S_n^{-\top}$ be the Cholesky decomposition of $G_n$ ($\det G_n\ne 0$) and let
$$
S_n=\begin{bmatrix}
1 & 0 & 0 & \ldots  & 0\\
s_{1,0} & 1 & 0 & \ldots  & 0\\
s_{2,0} & s_{2,1} & 1 & \ldots  & 0 \\
\vdots & \vdots &  \vdots & \ddots & \vdots \\
s_{n,0} & s_{n,1} & s_{n,2} & \ldots & 1
\end{bmatrix}
$$
be the explicit expression of the matrix $S_n$. It follows from Theorem \ref{th:choleskyorth} that the set of polynomials $\{P_0(x),P_1(x),\ldots, P_n(x)\}$ where
$$
P_k(x)=s_{k,0}+s_{k,1}x+\cdots+s_{k,k-1}x^{k-1}+x^k, \quad k=0,1,2,\ldots,n,
$$
form an orthogonal basis for $\Pi_n$ with respect to $\mathbf{u}$; that is, $\deg P_k = k$ and 
$$
\langle \mathbf{u}, P_j\,P_i\rangle = h_j\,\delta_{i,j}, \quad 0\leqslant i,j\leqslant n,
$$
with $h_j\ne 0$. Moreover, by Theorem \ref{th:choleskySn+1}  the polynomial
\begin{equation*}
P_{n+1}(x)= \frac{1}{\det G_n} \det\left[\begin{array}{ccc|c} & & & \mu_{n+1}\\
 & G_n & & \vdots \\
 & & & \mu_{2n+1} \\
 \hline 
1 & \ldots &  x^n & x^{n+1}\end{array} \right],
\end{equation*}
has degree exactly $n+1$, is orthogonal to every polynomial in $\Pi_n$ and
$$
h_{n+1}=\langle \mathbf{u}, P_{n+1}^2\rangle = \frac{\det G_{n+1}}{\det G_n}.
$$

\section{Characterizations of classical sequences}\label{characterizations}

Our goal for this section is to recast Theorem \ref{th:classical-char} in terms of classical sequences by shifting our point of view from the moment functional $\mathbf{u}$ to the Gram matrix $G_n$ associated with a bilinear form defined on $\Pi_n$. Recall that $G_n$ is a Hankel matrix (all of its antidiagonals are constant). Hence, we can say that this section deals with Hankel matrices with an additional structure: the entries of $G_n$ satisfy the recurrence relation \eqref{eq:ttr-moments}. In this way, we can extend the bilinear form to $\Pi_{n+1}$ by constructing a new Gram matrix $G_{n+1}$ by means of bordering $G_n$ with a new row and column whose entries are obtained with \eqref{eq:ttr-moments} from the entries of $G_n$. The resulting matrix $G_{n+1}$ will also be a Hankel matrix with the additional structure mentioned above. Consequently, it will be possible to prove by induction that the properties satisfied by $G_n$ are also satisfied by $G_{n+1}$.

The following matrices will play an important role in the sequel. 

\begin{definition}\label{def:matrixR}
Let $a,b,c,d$, and $e$ be real numbers such that $|a|+|b|+|c|>0$ and $n\,a+d\ne 0$ for $n\geqslant 0$. For $n\geqslant 1$, we define the $n\times (n+1)$ matrix $R_n$ recursively as follows:
$$
R_{n}=\left[ \begin{array}{ccccc|c} 
 & & & & & 0 \\
 & &  & R_{n-1} & & \vdots \\
 & & & & & 0\\
 \hline 
 0 & \ldots & 0 & (n-1)\,c & (n-1)\,b+e & (n-1)\,a+d
\end{array}\right], \quad R_1=\begin{bmatrix} e & d \end{bmatrix}.
$$
\end{definition}
Consider the differential operator $\mathcal{R}: \Pi\rightarrow \Pi$ defined as
$$
\mathcal{R}[p]=\phi(x)\,p'+\psi(x)\,p, \quad \forall p\in \Pi,
$$
where $\phi(x)=a\,x^2+b\,x+c$ and $\psi(x)=d\,x+e$. Observe that
$$
\mathcal{R}[x^n]\,=\,(n\,a+d)\,x^{n+1}+(n\,b+e)\,x^n+n\,c\,x^{n-1}, \quad n\geqslant 0.
$$
In this way, the matrix $R_{n+1}^{\top}$ is the matrix representation relative to the basis of monomials of $\mathcal{R}$ restricted to $\Pi_n$.

\bigskip

Let $\{\mu_n\}_{n\geqslant 0}$ be a sequence of real numbers. Define the vector of moments
$$
\mathtt{M}_n=\begin{bmatrix}\mu_0 & \mu_1 & \ldots & \mu_n\end{bmatrix}^{\top}, \quad n\geqslant 0.
$$
If $\{\mu_n\}_{n\geqslant 0}$ is pre-classical, then equation \eqref{eq:ttr-moments} for $\{\mu_0,\ldots,\mu_n\}$ can be written as
$$
R_n\,\mathtt{M}_n=\mathtt{0}.
$$
This implies that if $\mathbf{u}$ is the moment functional defined as $\mu_n:=\langle \mathbf{u}, x^n\rangle$, then by \eqref{eq:ttr-moments}, the following holds
$$
\left\langle \mathbf{u}, \mathcal{R}[x^n] \right\rangle =0, \quad n\geqslant 0.
$$

Now, consider the vector whose entries are the monomials in $\Pi_n$:
$$
\mathtt{X}_n=\begin{bmatrix}1 & x & \ldots & x^n \end{bmatrix}^{\top}.
$$
If we define the $n\times (n+1)$ matrix
$$
N_{n}=\left[ \begin{array}{c|c} 
N_{n-1} & \mathtt{0} \\
  \hline
\mathtt{0} & n
\end{array}\right], \quad N_1=\begin{bmatrix}0 & 1 \end{bmatrix}, \quad N_0 = 0,
$$
then
$$
\frac{d}{dx}\mathtt{X}_n\,=\,N_n^{\top}\,\mathtt{X}_{n-1}.
$$

\subsection{Bochner-type characterization}

For $n\geqslant 0$, let us introduce the $(n+1)\times (n+1)$ matrix $\mathtt{D}_n$ defined as
$$
    \mathtt{D}_0=0, \qquad \mathtt{D}_n:=R_n^{\top}N_n, \quad n\geqslant 1.
$$
It is possible to express $\mathtt{D}_n$ recursively as follows:
\begin{align*}
\mathtt{D}_{n}=\left[ \begin{array}{ccc|c} 
& & & 0 \\
&  & & \vdots \\
 & \mathtt{D}_{n-1} & & 0\\
 & & & n\,(n-1)\,c\\[6pt]
 & & & n\,((n-1)\,b+e)\\[6pt]
 \hline 
 0 & \ldots & 0 &   n\,((n-1)\,a+d)
\end{array}\right], \quad n\geqslant 1.
\end{align*}

Observe that $\mathtt{D}_n$ is the matrix representation relative the basis of monomials of the operator
\begin{equation}\label{def:operatorD}
\mathcal{D}[p]:=\mathcal{R}\left[\frac{d}{dx}p \right]=\phi(x)\,p''+\psi(x)\,p', \quad \forall p\in \Pi,
\end{equation}
restricted to $\Pi_n$.

In the following theorem, we show that $\mathtt{D}_n$ is a self-adjoint matrix with respect to the bilinear form $\mathfrak{B}_n$ given in Definition \ref{def:bilinearform}.

\begin{theorem}\label{th:selfadjoint}
    Let $\{\mu_n\}_{n\geqslant 0}$ be a classical sequence satisfying \eqref{eq:ttr-moments}, and let $\mathfrak{B}_n$ be the operator defined in Definition \ref{def:bilinearform}. Then, for $n\geqslant 0$, the matrix $\mathtt{D}_n$ satisfies
    \begin{equation}\label{eq:selfadjoint}
    \mathfrak{B}_n(\mathtt{D}_n\bar{u}, \bar{v})=\mathfrak{B}_n(\bar{u}, \mathtt{D}_n\bar{v}), \quad \bar{u},\bar{v}\in \mathbb{R}^{n+1}.
    \end{equation}
\end{theorem}
\begin{proof}
Observe that proving \eqref{eq:selfadjoint} is equivalent to proving
$$
\mathtt{D}_n^{\top}\,G_n=G_n\,\mathtt{D}_n.
$$
We prove this for $n\geqslant 0$ by induction.

It is obvious that $\mathtt{D}_0^{\top}\,G_0=G_0\,\mathtt{D}_0$. We also prove the case $n=1$ for the sake of clarity since it is the first non trivial case. We compute
$$
\mathtt{D}_1^{\top}\,G_1=\left[\begin{array}{c|c}
0 & 0 \\
\hline
e & d
\end{array}\right]\left[\begin{array}{c|c}
\mu_0 & \mu_1 \\
\hline
\mu_1 & \mu_2
\end{array}\right]=\left[\begin{array}{c|c}
\mathtt{D}_0^{\top} & 0 \\
\hline
e & d
\end{array}\right]\left[\begin{array}{c|c}
G_0 & \mu_1 \\
\hline
\mu_1 & \mu_2
\end{array}\right].
$$
Again, multiplying by blocks, we get
$$
\mathtt{D}_1^{\top}\,G_1=\left[\begin{array}{c|c}
\mathtt{D}_0^{\top}\,G_0 & \mathtt{D}_0^{\top}\,\mu_1 \\
\hline
d\,\mu_1+e\,\mu_0 & e\,\mu_1+d\,\mu_2
\end{array}\right]=\left[\begin{array}{c|c}
G_0\,\mathtt{D}_0^{\top} & 0 \\
\hline
d\,\mu_1+e\,\mu_0 & e\,\mu_1+d\,\mu_2
\end{array}\right].
$$
Since $\{\mu_n\}_{n\geqslant 0}$ is a classical sequence, by condition \eqref{eq:ttr-moments} with $n=0$, we have
$$
\mu_1\,\mathtt{D}_0=0=d\,\mu_1+e\,\mu_0.
$$
Therefore, we can write
$$
\mathtt{D}_1^{\top}\,G_1=\left[\begin{array}{c|c}
G_0\,\mathtt{D}_0^{\top} & e\,\mu_0+d\,\mu_1 \\
\hline
\mu_1\,\mathtt{D}_0 & e\,\mu_1+d\,\mu_2
\end{array}\right]=\left[\begin{array}{c|c}
G_0 & \mu_1 \\
\hline
\mu_1 & \mu_2
\end{array}\right]\left[\begin{array}{c|c}
\mathtt{D}_0 & e \\
\hline
0 & d
\end{array}\right]=G_1\,\mathtt{D}_1.
$$
This proves $\mathtt{D}_1^{\top}\,G_1=G_1\,\mathtt{D}_1$.

Now, suppose that $\mathtt{D}_k^{\top}\,G_k=G_k\,\mathtt{D}_k$ holds for some $k\geqslant 0$. We compute
\begin{align*}
&\mathtt{D}_{k+1}^{\top}\,G_{k+1}\\
&=\left[ \begin{array}{ccccc|c} 
 & & & & & 0 \\
 & &  & \mathtt{D}_{k}^{\top} & & \vdots \\
 & & & & & 0\\
 \hline 
 0 & \ldots & 0 & (k+1)\,k\,c & (k+1)\,(k\,b+e) & (k+1)\,(k\,a+d)
\end{array}\right]\\
&\qquad \times \left[\begin{array}{ccc|c}
 & & & \mu_{k+1} \\
 & G_{k} & &\vdots \\
 & & &\mu_{2k+1}\\
\hline 
\mu_{k+1} & \ldots & \mu_{2k+1} & \mu_{2k+2}
\end{array}\right].
\end{align*}
Multiplying by blocks, we get
$$
\mathtt{D}_{k+1}^{\top}\,G_{k+1}=\left[ \begin{array}{c|c}
\mathtt{D}_{k}^{\top}\,G_k & \bar{x}_k\\[3pt]
\hline
\\[-6pt]
\bar{y}_k^{\top} & \bar{z}_k
\end{array}
\right],
$$
where
$$
\bar{x}_k=\mathtt{D}_k^{\top}\begin{bmatrix}\mu_{k+1} \\ \vdots \\ \mu_{2k+1} \end{bmatrix}, 
$$
$$
\bar{y}_k^{\top}= \left[ (k+1)\,k\,c \ \ \ (k+1)\,(k\,b+e) \ \ \ (k+1)\,(k\,a+d)
\right]\left[\begin{array}{ccc}
\mu_{k-1} & \ldots & \mu_{2k-1}  \\
\mu_k & \ldots & \mu_{2k} \\
\hline 
\mu_{k+1} & \ldots & \mu_{2k+1} 
\end{array}\right],
$$
and
$$
\bar{z}_k=(k+1)\,k\,c\,\mu_{2k}+(k+1)\,(k\,b+e)\,\mu_{2k+1}+(k+1)\,(k\,a+d)\,\mu_{2k+2}=\bar{z}_k^{\top}.
$$
Since by our induction hypothesis we have
$$
\mathtt{D}_{k+1}^{\top}\,G_{k+1}=\left[ \begin{array}{c|c}
G_k\,\mathtt{D}_{k} & \bar{x}_k\\[3pt]
\hline
\\[-6pt]
\bar{y}_k^{\top} & \bar{z}_k^{\top}
\end{array}
\right],
$$
it is clear that our main efforts should focus on showing that $\bar{x}_k=\bar{y}_k$. Observe that
$$
\bar{x}_k=\left[ \begin{array}{ccccc|c} 
 & & & & & 0 \\
 & &  & \mathtt{D}_{k-1}^{\top} & & \vdots \\
 & & & & & 0\\
 \hline 
 0 & \ldots & 0 & k\,(k-1)\,c & k\,((k-1)\,b+e) & k\,((k-1)\,a+d)
\end{array}\right]\begin{bmatrix}\mu_{k+1} \\ \vdots \\ \mu_{2k+1} \end{bmatrix}.
$$
Then, the $k$-th entry of $\bar{x}_k$ is
$$
k\,(k-1)\,c\,\mu_{2k-1}+k\,((k-1)\,b+e)\,\mu_{2k}+k\,((k-1)\,a+d)\,\mu_{2k+1},
$$
while the $k$-th entry of $\bar{y}_k$ is
$$
(k+1)\,k\,c\,\mu_{2k-1}+(k+1)\,(k\,b+e)\,\mu_{2k}+(k+1)\,(k\,a+d)\,\mu_{2k+1},
$$
which means that the $k$-th entry of $\bar{y}_k-\bar{x}_k$ is
$$
(2k\,a+d)\,\mu_{2k+1}+(2k\,b+e)\,\mu_{2k}+2k\,c\,\mu_{2k-1}=0,
$$
where we have used condition \eqref{eq:ttr-moments}. 

Now, noticing that the last row of $\mathtt{D}_{k-1}^{\top}$ is
$$
\begin{bmatrix}
0 & \ldots & 0 & (k-1)\,(k-2)\,c & (k-1)\,((k-2)\,b+2) & (k-1)\,((k-2)\,a+d)
\end{bmatrix},
$$
we have that the $(k-1)$-th entry of $\bar{x}_k$ is
$$
(k-1)\,(k-2)\,c\,\mu_{2k-2}+(k-1)\,((k-2)\,b+2)\,\mu_{2k-1}+(k-1)\,((k-2)\,a+d)\,\mu_{2k},
$$
while the $(k-1)$-th entry of $\bar{y}_k$ is
$$
(k+1)\,k\,c\,\mu_{2k-2}+(k+1)\,(k\,b+e)\,\mu_{2k-1}+(k+1)\,(k\,a+d)\,\mu_{2k}.
$$
Then the $(k-1)$-th entry of $\bar{y}_k-\bar{x}_k$ is
$$
2\bigg[((2k-1)\,a+d)\,\mu_{2k}+((2k-1)\,b+e)\,\mu_{2k-1}+(2k-1)\,c\,\mu_{2k-2}\bigg]=0,
$$
where, again, we have used \eqref{eq:ttr-moments} with $n=k-1$. 

If we continue in this way, then for $i=0,1,\ldots,k$, the $(i+1)$-th entry of $\bar{y}_k-\bar{x}_k$ is
$$
(k+1-i)\bigg[((k+i)\,a+d)\mu_{k+1+i}+((k+i)\,b+e)\,\mu_{k+i}+(k+i)\,c\,\mu_{k-1+i} \bigg]=0.
$$
This means that $\bar{x}_k=\bar{y}_k$ and, consequently,
$$
\mathtt{D}_{k+1}^{\top}\,G_{k+1}=\left[ \begin{array}{c|c}
G_k\,\mathtt{D}_{k} & \bar{y}_k\\[3pt]
\hline
\\[-6pt]
\bar{x}_k^{\top} & \bar{z}_k^{\top}
\end{array}
\right] = G_{k+1}\,\mathtt{D}_{k+1},
$$
which proves that $\mathtt{D}_n^{\top}\,G_n=G_n\,\mathtt{D}_n$ holds for $n\geqslant 0$.
\end{proof}

Let us now consider the eigenvectors of $\mathtt{D}_n$. Suppose that $\bar{u},\bar{v}\in \mathbb{R}^{n+1}$ are eigenvectors corresponding to distinct eigenvalues $\lambda$ and $\overline{\lambda}$, respectively. Then,
$$
\lambda\,\mathfrak{B}_n(\bar{u}, \bar{v})=\mathfrak{B}_n(\lambda\,\bar{u},\bar{v})=\mathfrak{B}_n(\mathtt{D}_n\bar{u},\bar{v})=\mathfrak{B}_n(\bar{u},\mathtt{D}_n\bar{v})=\mathfrak{B}_n(\bar{u},\overline{\lambda}\,\bar{v})=\overline{\lambda}\,\mathfrak{B}_n(\bar{u},\bar{v}),
$$
which implies
$$
(\lambda-\overline{\lambda})\,\mathfrak{B}_n(\bar{u},\bar{v})=0.
$$
Since $\lambda\ne \overline{\lambda}$, we must have that $\mathfrak{B}_n(\bar{u},\bar{v})=0$ and, therefore, $\bar{u}$ and $\bar{v}$ are orthogonal with respect to $\mathfrak{B}_n$. We have already encountered the eigenvectors of $\mathtt{D}_n$, as the following theorem shows. This theorem is, in fact, a characterization of classical sequences in terms of $\mathtt{D}_n$.

\begin{theorem}\label{th:eighenvectors}
For $n\geqslant 0$, let $S_n^{-1}\,H_n\,S_n^{-\top}$ be the Cholesky factorization of $G_n$ and let $\bar{s}_{n,0},\bar{s}_{n,1},\ldots,\bar{s}_{n,n}$ denote the columns of $S_n^{\top}$. Then $\{\mu_n\}_{n\geqslant 0}$ is a classical sequence if and only if
$$
\mathtt{D}_n\,\bar{s}_{n,j}\,=\,\lambda_j\,\bar{s}_{n,j}, \quad n\geqslant 0, \quad 0\leqslant j \leqslant n,
$$
where $\lambda_j=j\,[(j-1)\,a+d]$.
\end{theorem}
\begin{proof}
Suppose that $\{\mu_n\}_{n\geqslant 0}$ is a classical sequence. For $n=0$, then $\mathtt{D}_0=0$ and $\bar{s}_{0,0}=1$. Thus it is obvious that $\mathtt{D}_0\,\bar{s}_{0,0}=\lambda_0\,\bar{s}_{0,0}$.

For $n=1$, 
$$
\mathtt{D}_1=\left[\begin{array}{c|c}
\mathtt{D}_0 & e \\
\hline 
0 & d
\end{array}\right], \quad \bar{s}_{1,0}=\left[ \begin{array}{c} \bar{s}_{0,0} \\ \hline 0 \end{array}\right].
$$
Then, multiplying by blocks, we have
$$
\mathtt{D}_1\,\bar{s}_{1,0}=\left[\begin{array}{c} \mathtt{D}_0\,\bar{s}_{0,0} \\ \hline 0 \end{array}\right]=\lambda_0\,\left[ \begin{array}{c} \bar{s}_{0,0} \\ \hline 0 \end{array}\right]=\lambda_0\,\bar{s}_{1,0}.
$$
Since $S_1^{\top}$ is upper triangular with $1$'s on its diagonal, then $\{\bar{s}_{1,0},\bar{s}_{1,1}\}$ constitutes a basis for $\mathbb{R}^2$. Then we can write
$$
\mathtt{D}_1\,\bar{s}_{1,1}=a_{1,1}\,\bar{s}_{1,1}+a_{1,0}\,\bar{s}_{1,0},
$$
for some constants $a_{1,1}$ and $a_{1,0}$. Using the orthogonality of the columns of $S_1^{\top}$ with respect to $\mathfrak{B}_1$ and Theorem \ref{th:selfadjoint}, we obtain
$$
a_{1,j}=\frac{\mathfrak{B}_1(\mathtt{D}_1\,\bar{s}_{1,1},\bar{s}_{1,j})}{h_j}=\frac{\mathfrak{B}_1(\bar{s}_{1,1},\mathtt{D}_1\,\bar{s}_{1,j})}{h_j}, \quad j=0,1.
$$
It follows that $a_{1,0}=0$ and, thus,
$$
\mathtt{D}_1\,\bar{s}_{1,1}=a_{1,1}\,\bar{s}_{1,1}=\begin{bmatrix}
    * \\ a_{1,1}
\end{bmatrix},
$$
where we have taken into account that $\bar{s}_{1,1}=[* \ 1]^{\top}$ (the value of the entry denoted by $*$ has no relevance). If we multiply by blocks, we get
$$
\mathtt{D}_1\,\bar{s}_{1,1}=\begin{bmatrix}
    * \\ d
\end{bmatrix},
$$
which implies that $a_{1,1}=d=\lambda_1$.

Now, suppose that 
$$
\mathtt{D}_k\,\bar{s}_{k,j}\,=\,\lambda_j\,\bar{s}_{k,j}, \quad 0\leqslant j \leqslant k,
$$
holds for some $k\geqslant 0$. Recall that
\begin{align*}
\mathtt{D}_{k+1}=\left[ \begin{array}{ccc|c} 
& & & 0 \\
&  & & \vdots \\
 & \mathtt{D}_{k} & & 0\\
 & & & (k+1)\,k\,c\\[6pt]
 & & & (k+1)\,(k\,b+e)\\[6pt]
 \hline 
 0 & \ldots & 0 &   (k+1)\,(k\,a+d)
\end{array}\right],
\end{align*}
and that, by Theorem \eqref{th:choleskySn+1},
$$
\bar{s}_{k+1,j}=\left[\begin{array}{c} 
\bar{s}_{k,j} \\
\hline 
0
\end{array}\right], \quad 0\leqslant j \leqslant k, \quad \text{and} \quad \bar{s}_{k+1,k+1}=\left[\begin{array}{c} 
* \\
\hline 
1
\end{array}\right],
$$
where the values of the entries denoted by $*$ are not relevant here. Multiplying by blocks and using the induction hypothesis, we get
$$
\mathtt{D}_{k+1}\,\bar{s}_{k+1,j}=\left[\begin{array}{c} 
\mathtt{D}_{k}\,\bar{s}_{k,j} \\
\hline 
0
\end{array}\right]=\lambda_j\,\bar{s}_{k+1,j}, \quad 0\leqslant j \leqslant k.
$$
Since $S_{k+1}^{\top}$ is an upper triangular matrix with $1$'s on its diagonal, its columns constitute a basis for $\mathbb{R}^{k+1}$. Then, we can write
$$
\mathtt{D}_{k+1}\,\bar{s}_{k+1,k+1}=\sum_{j=0}^{k+1}\,a_{k+1,j}\,\bar{s}_{k+1,j},
$$
where, by the orthogonality of the columns of $S_{k+1}^{\top}$ with respect to $\mathfrak{B}_{k+1}$ and Theorem \ref{th:selfadjoint}, we have
\begin{align*}
a_{k+1,j}&=\frac{\mathfrak{B}_{k+1}(\mathtt{D}_{k+1}\,\bar{s}_{k+1,k+1},\bar{s}_{k+1,j})}{h_j}\\
&=\frac{\mathfrak{B}_{k+1}(\bar{s}_{k+1,k+1},\mathtt{D}_{k+1}\,\bar{s}_{k+1,j})}{h_j}, \quad 0\leqslant j \leqslant k+1.
\end{align*}
It follows that
\begin{align*}
a_{k+1,j}=\,\lambda_j\,\frac{\mathfrak{B}_{k+1}(\bar{s}_{k+1,k+1},\bar{s}_{k+1,j})}{h_j}\,=\,0, \quad 0\leqslant j \leqslant k,
\end{align*}
and, consequently,
$$
\mathtt{D}_{k+1}\,\bar{s}_{k+1,k+1}=a_{k+1,k+1}\,\bar{s}_{k+1,k+1}=\left[\begin{array}{c} 
* \\
\hline 
a_{k+1,k+1}
\end{array}\right].
$$
Moreover, if we multiply by blocks, we get
$$
\mathtt{D}_{k+1}\,\bar{s}_{k+1,k+1}=\left[\begin{array}{c} 
* \\
\hline 
\lambda_{k+1}
\end{array}\right],
$$
which implies that $a_{k+1,k+1}=\lambda_{k+1}$. Then the sufficient condition is proved by the principle of induction.

\bigskip

Conversely, if $\mathfrak{B}_n$ is the bilinear form defined in Definition \ref{def:bilinearform}, then 
$$
\mathfrak{B}_n(\mathtt{D}_n\,\bar{s}_{n,j},\bar{s}_{n,k})\,=\,\mathfrak{B}_n(\bar{s}_{n,j},\mathtt{D}_n\,\bar{s}_{n,k}), \quad n\geqslant 0, \quad 0\leqslant j,k\leqslant n.
$$
Indeed,
\begin{align*}
\mathfrak{B}_n(\mathtt{D}_n\,\bar{s}_{n,j},\bar{s}_{n,k})-\mathfrak{B}_n(\bar{s}_{n,j},\mathtt{D}_n\,\bar{s}_{n,k})=\,(\lambda_j-\lambda_k)\,\mathfrak{B}_n(\bar{s}_{n,j},\bar{s}_{n,k})=\,(\lambda_j-\lambda_k)\,h_j\,\delta_{j,k},
\end{align*}
which vanishes for all $0\leqslant j,k\leqslant n$. This implies that
$$
S_n\,(\mathtt{D}_n^{\top}\,G_n-G_n\,\mathtt{D}_n)\,S_n^{\top}\,=\,\mathtt{0}, \quad n\geqslant 0,
$$
or, equivalently,
$$
\mathtt{D}_n^{\top}\,G_n-G_n\,\mathtt{D}_n\,=\,\mathtt{0}.
$$
Since $\mathtt{D}_n=R_n^{\top}N_n$, the first column of the above matrix identity reads $N_n^{\top}\,R_n\,\mathtt{M}_n=\mathtt{0}$, or, equivalently,
$$
k\,\left[\big((k-1)\,a+d \big)\,\mu_{k}+\big((k-1)\,b+e \big)\,\mu_{k-1}+(k-1)\,\mu_{k-2} \right]=0,
$$
for $n\geqslant 0$ and $0\leqslant  k \leqslant n$. It follows that $\{\mu_n\}_{n\geqslant 0}$ satisfies \eqref{eq:ttr-moments}; hence, it is pre-classical. Furthermore, $\{\mu_n\}_{n\geqslant 0}$ is classical since, otherwise, $G_n$ would not have a Cholesky factorization for some $n\geqslant 0$.
\end{proof}

The above results can be passed down to the operator $\mathcal{D}$ defined in \eqref{def:operatorD}. Let $\{\mu_n\}_{n\geqslant 0}$ be a classical sequence of real numbers, and let $\mathbf{u}$ be the moment functional defined as $\mu_n=\langle \mathbf{u}, x^n\rangle$, $n\geqslant 0$. Then \eqref{eq:selfadjoint} implies that
$$
\left\langle \mathbf{u}, \mathcal{D}[p]\,q \right\rangle \,=\, \left\langle \mathbf{u}, p\,\mathcal{D}[q]\right\rangle, \quad \forall p,q\in \Pi.
$$
That is, $\mathcal{D}$ is a self-adjoint operator on polynomials. Moreover, for $n\geqslant 0$, let $S_n^{-1}\,H_n\,S_n^{-\top}$ be the Cholesky factorization of $G_n$ and let $\bar{s}_{n,0},\bar{s}_{n,1},\ldots,\bar{s}_{n,n}$ denote the columns of $S_n^{\top}$. From Theorem \ref{th:eighenvectors} we deduce that the sequence of polynomials $\{P_n\}_{n\geqslant 0}$ with
$$
P_n(x)\,=\,\bar{s}_{n,n}^{\top}\,\mathtt{X}_n, \quad n\geqslant 0,
$$
are eigenfunctions of the operator $\mathcal{D}$. That is, 
$$
\mathcal{D}[P_n]\,=\,\lambda_n\,P_n, \quad n\geqslant 0,
$$
with $\lambda_n=n\,[(n-1)\,a+d]$. Note that $\{P_n\}_{n\geqslant 0}$ is a sequence of polynomials orthogonal with respect to $\mathbf{u}$.

\subsection{Hahn-type characterization}
Let $\{ \mu_n \}_{n \geqslant 0}$ be a sequence of real numbers and let $a,b,c\in \mathbb{R}$ such that $|a|+|b|+|c|>0$. We can define a new sequence $\{\sigma_{n}\}_{n\geqslant 0}$ as follows:
$$
\sigma_n = a \,\mu_{n+2} + b\, \mu_{n+1} + c\, \mu_n, \quad n\geqslant 0.
$$
Notice that if $\mathbf{u}$ is the moment functional defined as $\mu_n=\langle \mathbf{u}, x^n\rangle$, then $\{\sigma_n\}_{n\geqslant 0}$ is the sequence of moments of the functional given by $\mathbf{v}=\phi(x)\,\mathbf{u}$ where $\phi(x)=a\,x^2+b\,x+c$. Indeed, for $n\geqslant 0$,
$$
\langle \mathbf{v},x^n\rangle = \langle \mathbf{u},\phi\,x^n\rangle =\langle \mathbf{u}, a\,x^{n+2}+b\,x^{n+1}+c\,x^n\rangle = a \mu_{n+2} + b \mu_{n+1} + c \mu_n=\sigma_n.
$$
We denote by $\{ G_n^{(1)}\}_{n \geqslant 0}$ the sequence of $(n+1)\times (n+1)$ matrices with 
\begin{equation}\label{def:Gn1}
G_{0}^{(1)} = \sigma_0, \quad \text{and} \quad G_n^{(1)}= \left[\begin{array}{ccc|c}
 & & & \sigma_n \\
 & G_{n-1}^{(1)} & &\vdots \\
 & & &\sigma_{2n-1}\\
\hline 
\sigma_n & \ldots & \sigma_{2n-1} & \sigma_{2n}
\end{array}\right], \quad n\geqslant 1.
\end{equation}
The following theorem shows that the pre-classical character is inherited by $\{\sigma_n\}_{n\geqslant 0}$.

\begin{theorem}\label{th:sigmapreclassical}
If $\{ \mu_n \}_{n \geqslant 0}$ is pre-classical satisfying \eqref{eq:ttr-moments}, then $\{ \sigma_n \}_{n \geqslant 0}$ is pre-classical satisfying
$$
(n\,a+d_1)\,\sigma_{n+1}+(n\,b+e_1)\,\sigma_n+n\,c\,\sigma_{n-1}\,=\,0, \quad n\geqslant 0,
$$
where $d_1=d+2\,a$ and $e_1=e+b$. Moreover,
\begin{equation}\label{eq:sigmas}
\sigma_n \, =\, -(n\,a+d)\,\mu_{n+2} - (n\,b+e)\,\mu_{n+1}- n\,c\,\mu_n, \quad n\geqslant 0.
\end{equation}
\end{theorem}
\begin{proof}
For $n\geqslant 0$, we compute
\begin{align*}
&(n\,a+d_1)\, \sigma_{n+1} + (n\,b+e_1)\, \sigma_n + n\,c\, \sigma_{n-1}\\ 
&= a\,\big( [(n+2)a + d] \mu_{n+3} + [(n+2)b + e] \mu_{n+2} + (n+2)c \mu_{n+1}-b\,\mu_{n+2}-2\,c\,\mu_{n+1} \big) \\
&\ \  + b\,\big( [(n+1)a + d] \mu_{n+2} + [(n+1)b + e] \mu_{n+1} + (n+1)c\, \mu_{n}+a\,\mu_{n+2}-c\,\mu_n \big) \\
&\ \ \ + c\,\big( (n\,a + d)\, \mu_{n+1} + (n\,b + e)\, \mu_{n} + n\,c\, \mu_{n-1} +2\,a\,\mu_{n+1}+b\,\mu_n\big),
\end{align*}
where we have used $\sigma_n = a\, \mu_{n+2} + b\, \mu_{n+1} + c\, \mu_n$. By \eqref{eq:ttr-moments}, we have
\begin{align*}
(n\,a+d_1)&\, \sigma_{n+1} + (n\,b+e_1)\, \sigma_n + n\,c\, \sigma_{n-1}\\
&= -a\,b\, \mu_{n+2} - 2\,a\,c\,\mu_{n+1} +a\,b\, \mu_{n+2} -b\,c\, \mu_n + 2\,a\,c\, \mu_{n+1} + b\,c\, \mu_n = 0.
\end{align*}

Finally, \eqref{eq:sigmas} follows from the fact that \eqref{eq:ttr-moments} can be written as
$$
 a \,\mu_{n+2} + b\, \mu_{n+1} + c\, \mu_n\,=\,-(n\,a+d)\,\mu_{n+2} - (n\,b+e)\,\mu_{n+1}- n\,c\,\mu_n.
$$
\end{proof}

When $\{\mu_n\}_{n\geqslant 0}$ is a classical sequence of real numbers, the matrix $G_n^{(1)}$ satisfies an interesting and useful relation involving the matrices $G_n$ and $\mathtt{D}_n$.

\begin{prop}\label{prop:NGN}
Let $\{\mu_n\}_{n\geqslant 0}$ be a classical sequence of real numbers satisfying \eqref{eq:ttr-moments}. Then, for $n\geqslant 0$,
$$
N_{n+1}^{\top}\,G_n^{(1)}\,N_{n+1}\,=\,-\mathtt{D}_{n+1}^{\top}\,G_{n+1}.
$$
\end{prop}

\begin{proof}
We prove this theorem by induction. For $n=0$, on one hand we have
$$
N_1^{\top}\,G_0^{(1)}\,N_1\,=\,\begin{bmatrix} 0 \\ 1 \end{bmatrix}\,\sigma_0\,\begin{bmatrix} 0 & 1 \end{bmatrix} \,=\,\begin{bmatrix} 0 & 0 \\ 0 & \sigma_0\end{bmatrix}.
$$
On the other hand, we have
$$
-\mathtt{D}_{1}^{\top}\,G_{1}\,=\,\begin{bmatrix}
0 & 0 \\
-e & -d
\end{bmatrix}\,\begin{bmatrix}
\mu_0 & \mu_1 \\
\mu_1 & \mu_2
\end{bmatrix}\,=\,\begin{bmatrix} 0 & 0 \\ -e\,\mu_0-d\,\mu_1 & -e\,\mu_1-d\,\mu_2 \end{bmatrix}.
$$
Since $\{\mu_n\}_{n\geqslant 0}$ satisfies \eqref{eq:ttr-moments}, we have that $-e\,\mu_0-d\,\mu_1=0$ and, by \eqref{eq:sigmas}, $-e\,\mu_1-d\,\mu_2=\sigma_0$. Therefore,
$$
-\mathtt{D}_{1}^{\top}\,G_{1}\,=\,\begin{bmatrix} 0 & 0 \\ 0 & \sigma_0\end{bmatrix},
$$
which proves that $N_1^{\top}\,G_0^{(1)}\,N_1 = -\mathtt{D}_{1}^{\top}\,G_{1}$.

Now, suppose that $N_{k+1}^{\top}\,G_k^{(1)}\,N_{k+1}\,=\,-\mathtt{D}_{k+1}^{\top}\,G_{k+1}$ holds for $k\geqslant 0$. On one hand, we compute
\begin{align*}
N_{k+2}^{\top}\,&G_{k+1}^{(1)}\,N_{k+2}\\
&=\,\left[ \begin{array}{c|c} 
N_{k+1}^{\top} & \mathtt{0} \\[3pt]
  \hline
  \\[-6pt]
\mathtt{0} & k+2
\end{array}\right] \,\left[\begin{array}{ccc|c}
 & & & \sigma_{k+1} \\
 & G_{k}^{(1)} & &\vdots \\
 & & &\sigma_{2k+1}\\
\hline 
\sigma_{k+1} & \ldots & \sigma_{2k+1} & \sigma_{2k+2}
\end{array}\right]\, \left[ \begin{array}{c|c} 
N_{k+1} & \mathtt{0} \\[3pt]
  \hline
  \\[-6pt]
\mathtt{0} & k+2
\end{array}\right].
\end{align*}
Multiplying by blocks and using the induction hypothesis, we get
$$
N_{k+2}^{\top}\,G_{k+1}^{(1)}\,N_{k+2}\,=\,\left[\begin{array}{c|c} -\mathtt{D}_{k+1}^{\top}\,G_{k+1} & \bar{x}_k \\[3pt]
\hline 
\\[-6pt]
\bar{x}_k^{\top} & (k+2)^2\,\sigma_{2k+2}\end{array}\right],
$$
where
$$
\bar{x}_k\,=\,(k+2)\,N_{k+1}^{\top}\,\begin{bmatrix} \sigma_{k+1} \\ \vdots \\ \sigma_{2k+1} \end{bmatrix}.
$$
On the other hand,
\begin{align*}
&-\mathtt{D}_{k+2}^{\top}\,G_{k+2}\\
&=\left[ \begin{array}{ccccc|c} 
 & & & & & 0 \\
 & &  & -\mathtt{D}_{k+1}^{\top} & & \vdots \\
 & & & & & 0\\
 \hline 
 0 & \ldots & 0 & -(k+2)\,(k+1)\,c & -(k+2)\,[(k+1)\,b+e] & -(k+2)\,[(k+1)\,a+d]
\end{array}\right]\\
&\qquad \times \left[\begin{array}{ccc|c}
 & & & \mu_{k+2} \\
 & G_{k+1} & &\vdots \\
 & & &\mu_{2k+3}\\
\hline 
\mu_{k+2} & \ldots & \mu_{2k+3} & \mu_{2k+4}
\end{array}\right].
\end{align*}
After multiplying by blocks, we have
$$
-\mathtt{D}_{k+2}^{\top}\,G_{k+2}\,=\,\left[\begin{array}{c|c} -\mathtt{D}_{k+1}^{\top}\,G_{k+1} & \bar{y}_k \\[3pt]
\hline 
\\[-6pt]
\bar{z}_k & w_k\end{array}\right],
$$
where
\begin{align*}
&\bar{y}_k=-\mathtt{D}_{k+1}^{\top}\,\begin{bmatrix} \mu_{k+2} \\ \vdots \\ \mu_{2k+3} \end{bmatrix}, \\[6pt]
&\bar{z}_k=\begin{bmatrix}0 & \ldots & 0 & -(k+2)\,(k+1)\,c & -(k+2)\,[(k+1)\,b+e]\end{bmatrix}\,G_{k+1}\\
& \qquad \qquad \qquad \qquad \qquad \qquad \qquad \qquad -(k+2)\,[(k+1)\,a+d]\,\begin{bmatrix} \mu_{k+2} & \ldots & \mu_{2k+3}\end{bmatrix},\\[6pt]
&w_k= -(k+2)\bigg[(k+1)\,c\,\mu_{2k+2}+[(k+1)\,b+e]\,\mu_{2k+3} +[(k+1)\,a+d]\,\mu_{2k+4}\bigg].
\end{align*}
By Theorem \ref{th:selfadjoint}, we have that $\bar{z}_k=\bar{y}_k^{\top}$. Hence, our efforts should focus on showing that $\bar{x}_k=\bar{y}_k$ and $w_k=(k+2)^2\,\sigma_{2k+2}$. Let us start by proving the second identity.

Observe that
\begin{align*}
w_k=&-(k+2)\,\bigg[[(2k+2)\,a+d]\,\mu_{2k+4}+[(2k+2)\,b+e]\,\mu_{2k+3}+(2k+2)\,c\,\mu_{2k+2} \bigg]\\
&+(k+2)\,(k+1)\,(a\,\mu_{2k+4}+b\,\mu_{2k+3}+c\,\mu_{2k+2}).
\end{align*}
From \eqref{eq:sigmas} and the fact that $\sigma_{2k+2}=a\,\mu_{2k+4}+b\,\mu_{2k+3}+c\,\mu_{2k+2}$, we deduce that
$$
w_k=(k+2)\,\sigma_{2k+2}+(k+2)\,(k+1)\,\sigma_{2k+2}=(k+2)^2\,\sigma_{2k+2}.
$$

In order to prove that $\bar{x}_k=\bar{y}_k$, we note that
$$
\bar{x}_k\,=\,(k+2)\,\left[ \begin{array}{c|c} 
N_{k}^{\top} & \mathtt{0} \\[3pt]
  \hline
  \\[-6pt]
\mathtt{0} & k+1
\end{array}\right]\,\begin{bmatrix} \sigma_{k+1} \\ \vdots \\ \sigma_{2k+1} \end{bmatrix},
$$
and 
$$
\bar{y}_k=\left[ \begin{array}{ccccc|c} 
 & & & & & 0 \\
 & &  & -\mathtt{D}_{k}^{\top} & & \vdots \\
 & & & & & 0\\
 \hline 
 0 & \ldots & 0 & -(k+1)\,k\,c & -(k+1)\,(k\,b+e) & -(k+1)\,(k\,a+d)
\end{array}\right]\,\begin{bmatrix} \mu_{k+2} \\ \vdots \\ \mu_{2k+3} \end{bmatrix}.
$$
Then, the last entry of $\bar{x}_k-\bar{y}_k$ is
\begin{align*}
&(k+2)\,(k+1)\,\sigma_{2k+1}+(k+1)\,[k\,c\,\mu_{2k+1}+(k\,b+e)\,\mu_{2k+2}+(k\,a+d)\,\mu_{2k+3}]\\
&=(k+2)\,(k+1)\,\sigma_{2k+1}-(k+1)\,\sigma_{2k+1}-(k+1)^2\,\sigma_{2k+1}=0,
\end{align*}
where we have used \eqref{eq:sigmas} and the fact that $\sigma_{2k+1}=a\,\mu_{2k+3}+b\,\mu_{2k+2}+c\,\mu_{2k+1}$. In general, the $i$-th entry of $\bar{x}_k-\bar{y}_k$, with $1\leqslant i \leqslant k+2$, is
\begin{align*}
&(i-1)\,\bigg[(k+2)\,\sigma_{k+i-1}+[(i-2)\,c\,\mu_{k+i-1}+[(i-2)\,b+e]\,\mu_{k+i}+[(i-2)\,a+d]\,\mu_{k+1+i}]\bigg]\\
&=(i-1)\,\bigg[(k+2)\,\sigma_{k+i-1}-\sigma_{k+i-1}-(k+1)\,\sigma_{k+i-1} \bigg]=0,
\end{align*}
where we have used \eqref{eq:sigmas} and the fact that $\sigma_{k+i-1}=a\,\mu_{k+1+i}+b\,\mu_{k+i}+c\,\mu_{k+i-1}$. This proves that $\bar{x}_k-\bar{y}_k=\mathtt{0}$ and, in turn, that $N_{k+2}^{\top}\,G_{k+1}^{(1)}\,N_{k+2}\,=\,-\mathtt{D}_{k+2}^{\top}\,G_{k+2}$.

It follows from the Principle of Induction that $N_{n+1}^{\top}\,G_n^{(1)}\,N_{n+1}\,=\,-\mathtt{D}_{n+1}^{\top}\,G_{n+1}$ holds for $n\geqslant 0$.
\end{proof}

Now, let $\{\mu_n\}_{n\geqslant 0}$ be a classical sequence of real numbers. For $n\geqslant 1$, let $S_{n}^{-1}\,H_{n}\,S_{n}^{-\top}$ be the Cholesky factorization of $G_{n}$ and let $\bar{s}_{n,0},\bar{s}_{n,1},\ldots,\bar{s}_{n,n}$ denote the columns of $S_{n}^{\top}$. Consider the set $\{ \bar{s}_{n,0}^{(1)}, \ldots, \bar{s}_{n,n}^{(1)}\}$ of vectors in $\mathbb{R}^{n+1}$ with
$$
\bar{s}_{n,j}^{(1)}=\frac{1}{j+1}\,N_{n+1}\,\bar{s}_{n+1,j+1}, \quad 0\leqslant j \leqslant n.
$$
Observe that $\{ \bar{s}_{n,0}^{(1)}, \ldots, \bar{s}_{n,n}^{(1)}\}$ constitutes a basis for $\mathbb{R}^{n+1}$. Furthermore, since $S_{n}^{\top}$ is a unit upper triangular matrix, the $(n+1)\times (n+1)$ matrix $S_{n,1}^{\top}$ defined as
\begin{equation}\label{def:Sn1}
S_{n,1}^{\top}=\left[\bar{s}_{n,0}^{(1)}\ \bar{s}_{n,1}^{(1)}\, \ldots\, \bar{s}_{n,n}^{(1)} \right],
\end{equation}
is also a unit upper triangular matrix. We show that $G_n^{(1)}$ defined in \eqref{def:Gn1} admits a Cholesky factorization with $S_{n,1}^{-1}$  its triangular matrix factor.

\begin{theorem}\label{th:choleskysigma}
Let $\{\mu_n\}_{n\geqslant 0}$ be a classical sequence of real numbers satisfying \eqref{eq:ttr-moments}. For $n\geqslant 0$, let $S_{n}^{-1}\,H_{n}\,S_{n}^{-\top}$ be the Cholesky factorization of $G_{n}$. Then $G_n^{(1)}$ admits the Cholesky factorization given by
$$
G_n^{(1)}= S_{n,1}^{-1}\,H_{n,1}\,S_{n,1}^{-\top},
$$
where $S_{n,1}^{\top}$ is the matrix defined in \eqref{def:Sn1} and $H_{n,1}=\textnormal{diag}[h_{0}^{(1)}, \ldots, h_{n}^{(1)} ]$ with
$$
h^{(1)}_{j}=-\frac{\lambda_{j+1}}{(j+1)^2}\,h_{j+1}, \quad j\geqslant 0,
$$
and $\lambda_j=j\,[(j-1)\,a+d]$.
\end{theorem}
\begin{proof}
For $0\leqslant j,k \leqslant n$, we compute
$$
\left(\bar{s}_{n,j}^{(1)}\right)^{\top}\,G_n^{(1)}\,\bar{s}_{n,k}^{(1)}\,=\,\frac{1}{(j+1)\,(k+1)}\,\bar{s}_{n+1,j+1}^{\top}\,N_{n+1}^{\top}\,G_n^{(1)}\,N_{n+1}\,\bar{s}_{n+1,k+1}.
$$
Using Theorem \ref{th:eighenvectors} and Proposition \ref{prop:NGN}, we obtain
\begin{align*}
\left(\bar{s}_{n,j}^{(1)}\right)^{\top}\,G_n^{(1)}\,\bar{s}_{n,k}^{(1)}\,&=\,-\frac{1}{(j+1)\,(k+1)}\,\bar{s}_{n+1,j+1}^{\top}\,\mathtt{D}_{n+1}^{\top}\,G_{n+1}\,\bar{s}_{n+1,k+1}\\
&=-\frac{\lambda_{j+1}}{(j+1)\,(k+1)}\,\bar{s}_{n+1,j+1}^{\top}\,G_{n+1}\,\bar{s}_{n+1,k+1}\\
&=-\frac{\lambda_{j+1}}{(j+1)\,(k+1)}\,\mathcal{B}_{n+1}(\bar{s}_{n+1,j+1},\,\bar{s}_{n+1,k+1})\\
&=-\frac{\lambda_{j+1}}{(j+1)^2}\,h_{j+1}\,\delta_{j,k}.
\end{align*}
This implies that $S_{n,1}\,G_n^{(1)}\,S_{n,1}^{\top}= H_{n,1}$ and, hence, $G_n^{(1)}= S_{n,1}^{-1}\,H_{n,1}\,S_{n,1}^{-\top}$.
\end{proof}

\begin{corollary}\label{coro:sigmaclassical}
    If $\{\mu_n\}_{n\geqslant 0}$ is a classical sequence, then so is $\{\sigma_n\}_{n\geqslant 0}$.
\end{corollary}

\begin{proof}
By Theorem \eqref{th:sigmapreclassical}, $\{\sigma_n\}_{n\geqslant 0}$ is a pre-classical sequence. 

Now, we must show that $\det G_n^{(1)}\ne 0$ for $n\geqslant 0$ (Definition \ref{def:classical}). From Theorem \ref{th:choleskysigma} we deduce that
$$
\det G_n^{(1)}=\det H_{n,1}=h^{(1)}_{0}\cdots h^{(1)}_{n}, \quad n\geqslant 0,
$$
with
$$
h^{(1)}_{j}=-\frac{\lambda_{j+1}}{(j+1)^2}\,h_{j+1}, \quad j\geqslant 0 ,
$$
and $\lambda_j=j\,[(j-1)\,a+d]$. Since $\{\mu_n\}_{n\geqslant 0}$ is classical, then we have
$$
h_{n} \ne 0 \quad \text{and} \quad n\,a+d\ne 0, \quad n\geqslant 0, 
$$
(see equality \eqref{eq:hn+1} and Definition \ref{def:classical sequence}). This implies that $h^{(1)}_j\ne 0$ for $j\geqslant 0$. Therefore $\det G_n^{(1)}\ne 0$ for $n\geqslant 0$ and, thus, $\{\sigma_n\}_{n\geqslant 0}$ is classical.
\end{proof}

 Theorem \ref{th:choleskysigma} implies that for $n\geqslant 0$, the columns of $S_{n,1}^{\top}$ constitute an orthogonal basis for $\mathbb{R}^{n+1}$ with respect to the bilinear form associated with $\{\sigma_n\}_{n\geqslant 0}$ (see Definition \ref{def:bilinearform}), which we  denote by
$$
\mathfrak{B}^{(1)}_n(\bar{u},\bar{v})\,=\,\bar{u}^{\top}\,G_n^{(1)}\,\bar{v}, \quad \forall \bar{u},\bar{v}\in \mathbb{R}^{n+1}.
$$
We are ready for the following characterizations of classical sequences.

\begin{theorem}\label{th:Hahn-type}
    Let $\{\mu_n\}_{n\geqslant 0}$ be a sequence of real numbers such that $\det G_n\ne 0$ for $n\geqslant 0$. Let $S_{n}^{-1}\,H_{n}\,S_{n}^{-\top}$ be the Cholesky factorization of $G_{n}$ and let $\bar{s}_{n,0},\bar{s}_{n,1},\ldots,\bar{s}_{n,n}$ denote the columns of $S_{n}^{\top}$. Then $\{\mu_n\}_{n\geqslant 0}$ is classical if and only if there are real numbers $a,b,c$ satisfying 
$$
|a|+|b|+|c|>0,
$$
such that the set $\{ \bar{s}_{n,0}^{(1)}, \ldots, \bar{s}_{n,n}^{(1)}\}$ of vectors in $\mathbb{R}^{n+1}$ with
$$
\bar{s}_{n,j}^{(1)}=\frac{1}{j+1}\,N_{n+1}\,\bar{s}_{n+1,j+1}, \quad 0\leqslant j \leqslant n,
$$
constitutes an orthogonal basis for $\mathbb{R}^{n+1}$ with respect to the bilinear form associated with $\{\sigma_n\}_{n\geqslant 0}$, where
$$
\sigma_n\,=\,a\,\mu_{n+2}+b\,\mu_{n+1}+c\,\mu_n, \qquad n\geqslant 0.
$$
\end{theorem}

\begin{proof}
If $\{\mu_n\}_{n\geqslant 0}$ is classical, then it follows from Theorem \ref{th:choleskysigma} that
$$
S_{n,1}\,G_n^{(1)}\,S_{n,1}^{\top}\,=\,H_{n,1}, \quad n\geqslant 0,
$$
where
$$
S_{n,1}^{\top}=\left[\bar{s}_{n,0}^{(1)}\ \bar{s}_{n,1}^{(1)}\, \ldots\, \bar{s}_{n,n}^{(1)} \right],
$$
and $H_{n,1}=\textnormal{diag}[h_{0}^{(1)}, \ldots, h_{n}^{(1)} ]$ with
$$
h^{(1)}_{j}=-\frac{\lambda_{j+1}}{(j+1)^2}\,h_{j+1}, \quad j\geqslant 0,
$$
and $\lambda_j=j\,[(j-1)\,a+d]$. This implies that
$$
\mathfrak{B}_n^{(1)}(\bar{s}_{n,i}^{(1)},\bar{s}_{n,j}^{(1)})=h^{(1)}_j\,\delta_{i,j}, \quad 0\leqslant i,j\leqslant n.
$$
This proves the necessary condition.

Conversely, for $n\geqslant 0$, on one hand we have
$$
\mathfrak{B}_n^{(1)}(\bar{s}_{n,k}^{(1)},\bar{s}_{n,0}^{(1)})=h_0^{(1)}\,\delta_{k,0} 
$$
or, equivalently,
$$
\left(\bar{s}_{n,k}^{(1)}\right)^{\top}\,G_n^{(1)}\,\bar{s}_{n,0}^{(1)}\,=\,\frac{1}{k+1}\,\bar{s}_{n+1,k+1}^{\top}\,N_{n+1}^{\top}\,G_n^{(1)}\,\bar{s}_{n,0}^{(1)}=-\lambda_{1}\,h_{1}\,\delta_{k,0}.
$$
Using the fact that $s_{n,0}^{(1)}=\bar{e}_0$ where, recall, $\bar{e}_0$ is the first column of the identity matrix of order $n+1$, we write
$$
\bar{s}_{n+1,k+1}^{\top}\,N_{n+1}^{\top}\,\begin{bmatrix}
\sigma_0 \\ \sigma_1 \\ \vdots \\ \sigma_n
\end{bmatrix}=-(k+1)\,\lambda_{1}\,h_{1}\,\delta_{k,0}.
$$
On the other hand
$$
\mathfrak{B}_{n+1}(\bar{s}_{n+1,k+1},\bar{s}_{n+1,1})=\bar{s}_{n+1,k+1}^{\top}\,G_{n+1}\,\bar{s}_{n+1,1}= h_1\,\delta_{k,0}=h_1\,(k+1)\,\delta_{k,0}.
$$
Therefore,
$$
\bar{s}_{n+1,k+1}^{\top}\,N_{n+1}^{\top}\,\begin{bmatrix}
\sigma_0 \\ \sigma_1 \\ \vdots \\ \sigma_n
\end{bmatrix}=-\lambda_{1}\,\bar{s}_{n+1,k+1}^{\top}\,G_{n+1}\,\bar{s}_{n+1,1},
$$
and, thus,
$$
\bar{s}_{n+1,k+1}^{\top}\left(N_{n+1}^{\top}\,\begin{bmatrix}
\sigma_0 \\ \sigma_1 \\ \vdots \\ \sigma_n
\end{bmatrix}+\lambda_{1}\,G_{n+1}\,\bar{s}_{n+1,1} \right)=0, \quad 0\leqslant k \leqslant n.
$$
Since the first row of $N_{n+1}^{\top}$ is $[0 \ 0 \ldots 0]$ and $\bar{s}_{n+1,0}^{\top}\,G_{n+1}\,\bar{s}_{n+1,1}=0$, we have
$$
S_{n+1}^{\top}\left(N_{n+1}^{\top}\,\begin{bmatrix}
\sigma_0 \\ \sigma_1 \\ \vdots \\ \sigma_n
\end{bmatrix}+\lambda_{1}\,G_{n+1}\,\bar{s}_{n+1,1} \right)=\mathtt{0},
$$
and, consequently, 
\begin{equation}\label{eq:hahn}
N_{n+1}^{\top}\,\begin{bmatrix}
\sigma_0 \\ \sigma_1 \\ \vdots \\ \sigma_n
\end{bmatrix}+\lambda_{1}\,G_{n+1}\,\bar{s}_{n+1,1} =\mathtt{0},
\end{equation}
where we have used the fact that $S_n^{\top}$ is an invertible matrix. If
$$
\bar{s}_{n+1,1}=[s_{1,0} \ 1 \ 0 \ldots 0]^{\top},
$$
then let $d=\lambda_1$ and $e=\lambda_1\,s_{1,0}$. Observe that by Theorem \ref{th:choleskySn+1}, $s_{1,0}$ is independent of $n$. In this way, the entries of \eqref{eq:hahn} read
\begin{align*}
&d\,\mu_1+e\,\mu_0 = 0,\\
&k\,\sigma_{k-1}+d\,\mu_{k+1}+e\,\mu_k=(k\,a+d)\,\mu_{k+1}+(k\,b+e)\,\mu_k+k\,c\,\mu_{k-1}=0, \quad 1\leqslant k \leqslant n,
\end{align*}
which proves that $\{\mu_n\}_{n\geqslant 0}$ satisfies \eqref{eq:ttr-moments} and, thus, $\{\mu_n\}_{n\geqslant 0}$ is classical.
\end{proof}

Corollary \ref{coro:sigmaclassical} shows that the classical character of $\{\mu_n\}_{n\geqslant 0}$ is inherited by $\{\sigma_n\}_{n\geqslant 0}$ allowing us to apply all of our previous results about classical sequences to $\{\sigma_n\}_{n\geqslant 0}$ and $\{G_n^{(1)}\}_{n\geqslant 0}$. Therefore, we can define a new sequence $\{\sigma_{n}^{(2)}\}_{n\geqslant 0}$ as follows:
$$
\sigma_n^{(2)} = a \,\sigma_{n+2} + b\, \sigma_{n+1} + c\, \sigma_n, \quad n\geqslant 0.
$$
We denote by $\{ G_n^{(2)}\}_{n \geqslant 0}$ the sequence of $(n+1)\times (n+1)$ matrices with 
\begin{equation*}\label{def:Gn2}
G_{0}^{(2)} = \sigma_0^{(2)}, \quad \text{and} \quad G_n^{(2)}= \left[\begin{array}{ccc|c}
 & & & \sigma_n^{(2)} \\
 & G_{n-1}^{(2)} & &\vdots \\
 & & &\sigma_{2n-1}^{(2)}\\
\hline 
\sigma_n^{(2)} & \ldots & \sigma_{2n-1}^{(2)} & \sigma_{2n}^{(2)}
\end{array}\right], \quad n\geqslant 1.
\end{equation*}
By Theorem \ref{th:sigmapreclassical}, $\{\sigma_{n}^{(2)}\}_{n\geqslant 0}$ is pre-classical satisfying
$$
(n\,a+d_2)\,\sigma_{n+1}^{(2)}+(n\,b+e_2)\,\sigma_n^{(2)}+n\,c\,\sigma_{n-1}^{(2)}\,=\,0, \quad n\geqslant 0,
$$
where $d_2=d+4\,a$ and $e_2=e+2\,b$. Moreover, if $\det G_n^{(2)}\ne 0$ for $n\geqslant 0$, then  $\{\sigma_{n}^{(2)}\}_{n\geqslant 0}$ is classical and, by Theorem \ref{th:Hahn-type}, the set $\{ \bar{s}_{n,0}^{(2)}, \ldots, \bar{s}_{n,n}^{(2)}\}$ of vectors in $\mathbb{R}^{n+1}$ with
$$
\bar{s}_{n,j}^{(2)}=\frac{1}{j+1}\,N_{n+1}\,\bar{s}_{n+1,j+1}^{(1)}, \quad 0\leqslant j \leqslant n,
$$
constitutes an orthogonal basis for $\mathbb{R}^{n+1}$ with respect to the bilinear form $\mathfrak{B}_n^{(2)}$ associated with $\{\sigma_n^{(2)}\}_{n\geqslant 0}$, that is,
$$
\mathfrak{B}^{(2)}_n(\bar{s}_{n,j}^{(2)},\bar{s}_{n,k}^{(2)})=h_j^{(2)}\,\delta_{j,k}, \quad 0\leqslant j, k \leqslant n,
$$
where
$$
h^{(2)}_{j}=-\frac{\lambda_{j+1}^{(1)}}{(j+1)^2}\,h_{j+1}^{(2)}, \quad j\geqslant 0 ,
$$
and $\lambda_j^{(1)}=j\,[(j-1)\,a+d_1]$.

Iterating this idea, we obtain the following result.

\begin{corollary}\label{th:higher-Hahn-type}
Let $\{\mu_n\}_{n\geqslant 0}$ be a classical sequence of real numbers. Let $S_{n}^{-1}\,H_{n}\,S_{n}^{-\top}$ be the Cholesky factorization of $G_{n}$ and let $\bar{s}_{n,0},\bar{s}_{n,1},\ldots,\bar{s}_{n,n}$ denote the columns of $S_{n}^{\top}$. For each $k\geqslant 1$, define the sequence of real numbers $\{\sigma_n^{(k)}\}_{n\geqslant 0}$ by
$$
\sigma_n^{(k)}\, =\, a\, \sigma_{n+2}^{(k-1)} + b\, \sigma_{n+1}^{(k-1)} + c\, \sigma_n^{(k-1)}, \quad n\geqslant 0,
$$
where $\sigma_n^{(0)}=\mu_n$ for $n\geqslant 0$. Then $\{ \sigma_n^{(k)} \}_{n \geqslant 0}$ is classical satisfying
$$
(n\,a+d_k)\,\sigma_{n+1}^{(k)}+(n\,b+e_k)\,\sigma_n^{(k)}+n\,c\,\sigma_{n-1}^{(k)}\,=\,0, \quad n\geqslant 0,
$$
where $d_k=d+2\,k\,a$ and $e_k=e+k\,b$. Moreover, the set $\{ \bar{s}_{n,0}^{(k)}, \ldots, \bar{s}_{n,n}^{(k)}\}$ of vectors in $\mathbb{R}^{n+1}$ with
$$
\bar{s}_{n,j}^{(k)}=\frac{1}{j+1}\,N_{n+1}\,\bar{s}_{n+1,j+1}^{(k-1)}, \quad 0\leqslant j \leqslant n,
$$
where $s_{n,j}^{(0)}=\bar{s}_{n,j}$, constitutes an orthogonal basis for $\mathbb{R}^{n+1}$ with respect to the bilinear form associated with $\{\sigma_n^{(k)}\}_{n\geqslant 0}$.
\end{corollary}

Observe that the vectors $\{\bar{s}_{n,0}^{(k)}, \ldots, \bar{s}_{n,n}^{(k)}\}$ in Corollary \ref{th:higher-Hahn-type} can be written in terms of the vectors $\bar{s}_{n,0},\bar{s}_{n,1},\ldots,\bar{s}_{n,n}$ as follows: for each $k\geqslant 1$,
\begin{equation}\label{eq:longsn}
\bar{s}_{n,j}^{(k)}= \frac{1}{(j+1)_k}N_{n+1}\,N_{n+2}\cdots N_{n+k}\,\bar{s}_{n+k,j+k}, \quad 0\leqslant j \leqslant n,
\end{equation}
where $(\nu)_k=\nu\,(\nu+1)\cdots (\nu+k-1)$, $(\nu)_0=1$, denotes the Pochhammer symbol. If $\{ G_n^{(k)}\}_{n \geqslant 0}$ dentoes the sequence of $(n+1)\times (n+1)$ matrices with 
\begin{equation}\label{def:Gnk}
G_{0}^{(k)} = \sigma_0^{(k)}, \quad \text{and} \quad G_n^{(k)}= \left[\begin{array}{ccc|c}
 & & & \sigma_n^{(k)} \\
 & G_{n-1}^{(k)} & &\vdots \\
 & & &\sigma_{2n-1}^{(k)}\\
\hline 
\sigma_n^{(k)} & \ldots & \sigma_{2n-1}^{(k)} & \sigma_{2n}^{(k)}
\end{array}\right], \quad n\geqslant 1,
\end{equation}
then the orthogonality of $\{\bar{s}_{n,0}^{(k)}, \ldots, \bar{s}_{n,n}^{(k)}\}$ with respect to the bilinear form $\mathfrak{B}_n^{(k)}$ associated with $\{\sigma_n^{(k)}\}_{n\geqslant 0}$ is given by
\begin{equation}\label{eq:orthogonalityk}
\mathfrak{B}^{(k)}_n(\bar{s}_{n,j}^{(k)},\bar{s}_{n,i}^{(k)})\,=\,(\bar{s}_{n,j}^{(k)})^{\top}\,G_n^{(k)}\,\bar{s}_{n,i}^{(k)} \,=\,h_j^{(k)}\,\delta_{j,i}, \quad 0\leqslant i,j \leqslant n,
\end{equation}
where
$$
h^{(k)}_{j}=-\frac{\lambda_{j+1}^{(k-1)}}{(j+1)^2}\,h_{j+1}^{(k-1)}, \quad j\geqslant 0 ,
$$
and $\lambda_j^{(k)}=j\,[(j-1)\,a+d_k]$. Note that we can write
$$
h^{(k)}_{j}=(-1)^k\frac{\prod_{i=0}^{k-1}\lambda_{j+k-i}^{(i)}}{[(j+1)_k]^2}\,h_{j+k}, \quad j\geqslant 0 .
$$
Moreover, \eqref{eq:orthogonalityk} implies the Cholesky factorization of $G_n^{(k)}$:
$$
S_{n,k}\,G_n^{(k)}\,S_{n,k}^{\top}\,=\,H_{n,k}, \quad n\geqslant 0,
$$
where
$$
S_{n,k}^{\top}=\left[\bar{s}_{n,0}^{(k)}\ \bar{s}_{n,1}^{(k)}\, \ldots\, \bar{s}_{n,n}^{(k)} \right],
$$
and $H_{n,k}=\textnormal{diag}[h_{0}^{(k)}, \ldots, h_{n}^{(k)} ]$.

Let us reformulate the above results in terms of polynomials and moment functionals. Let $\{\mu_n\}_{n\geqslant 0}$ be a classical sequence of real numbers satisfying \eqref{eq:ttr-moments}, and let $\mathbf{u}$ be the moment functional defined as $\mu_n=\langle \mathbf{u}, x^n\rangle$, $n\geqslant 0$.  For each $k\geqslant 1$, the sequence of real numbers $\{\sigma_n^{(k)}\}_{n\geqslant 0}$ defined by
$$
\sigma_n^{(k)}\, =\, a\, \sigma_{n+2}^{(k-1)} + b\, \sigma_{n+1}^{(k-1)} + c\, \sigma_n^{(k-1)}, \quad n\geqslant 0,
$$
where $\sigma_n^{(0)}=\mu_n$ for $n\geqslant 0$, is the sequence of moments of the functional given by $\mathbf{v}_k\,=\,\phi^k\,\mathbf{u}$ where $\phi(x)=a\,x^2+b\,x+c$. Moreover, for $n\geqslant 0$, let $S_n^{-1}\,H_n\,S_n^{-\top}$ be the Cholesky factorization of $G_n$ and let $\bar{s}_{n,0},\bar{s}_{n,1},\ldots,\bar{s}_{n,n}$ denote the columns of $S_n^{\top}$. For $n\geqslant 0$, Corollary \ref{th:higher-Hahn-type} implies that the polynomials $\{Q_{0,k}(x), Q_{1,k}(x),\ldots,Q_{n,k}(x)\}$ given by 
$$
Q_{j,k}(x)\,=\,(\bar{s}_{n,j}^{(k)})^{\top}\,\mathtt{X}_n, \quad 0\leqslant j \leqslant n,
$$
satisfy $\langle \mathbf{v}_k, Q_{i,k}\,Q_{j,k}\rangle = h_j^{(k)}\,\delta_{i,j}$. Therefore, $\{Q_{0,k}(x), Q_{1,k}(x),\ldots,Q_{n,k}(x)\}$ constitutes an orthogonal basis for $\Pi_n$ with respect to $\mathbf{v}_k$. Furthermore, Theorem \ref{th:choleskySn+1} allows us to write
$$
Q_{n,k}(x)\,=\,(\bar{s}_{n,n}^{(k)})^{\top}\,\mathtt{X}_n, \quad n\geqslant 0,
$$
and, in this way, $\{Q_{n,k}(x)\}_{n\geqslant 0}$ is an MOPS associated with $\mathbf{v}_k$. We note that if 
$$
P_n(x)\,=\,\bar{s}_{n,n}^{\top}\,\mathtt{X}_n, \quad n\geqslant 0,
$$
then, by \eqref{eq:longsn}, we have
$$
Q_{n,k}(x)\,=\,\frac{P_{n+k}^{(k)}(x)}{(n+1)_k}, \quad n\geqslant 0,
$$
where $P_n^{(k)}(x)$ denotes the $k$-th order derivative of $P_n(x)$.

\subsection{First structure relation} 
For $n\geqslant 1$ and given real numbers $a,b$, and $c$, we define the $(n+3)\times (n+1)$ matrices
\begin{equation}\label{def:matrizPhi}
\Phi_n=\left[\begin{array}{c|c} 
 & \mathtt{0} \\
\Phi_{n-1} & c \\
 & b \\
\hline 
\mathtt{0} & a
\end{array}\right], \quad \Phi_0=\begin{bmatrix}
    c \\ b \\ a
\end{bmatrix}.
\end{equation}

\begin{lemma}
    Let $\{\mu_n\}_{n\geqslant 0}$ and $\{\sigma_n\}_{n\geqslant 0}$ be sequences of real numbers satisfying
    $$
    \sigma_n=a\,\mu_{n+2}+b\,\mu_{n+1}+c\,\mu_n, \quad n\geqslant 0,
    $$
    where $a,b,c\in \mathbb{R}$.  Then, for $n\geqslant 2$, 
    \begin{equation}\label{eq:GPhi}
    G_n^{(1)}\,\begin{bmatrix}
        \bar{v} \\ 0 \\ 0
    \end{bmatrix}\,=\,G_n\,\Phi_{n-2}\,\bar{v}, \quad \forall\, \bar{v}\in \mathbb{R}^{n-1}.
    \end{equation}
\end{lemma}
\begin{proof}
We use induction to prove this result. For $n=2$, observe that, for all $v\in \mathbb{R}$,
$$
G_2^{(1)}\,\begin{bmatrix}
v \\ 0 \\ 0\end{bmatrix}=\begin{bmatrix}
    \sigma_0 \\ \sigma_1 \\ \sigma_2
\end{bmatrix}v=a\,v\,\begin{bmatrix}
    \mu_2 \\ \mu_3 \\ \mu_4
\end{bmatrix}+b\,v\,\begin{bmatrix}
\mu_1 \\ \mu_2 \\ \mu_3
\end{bmatrix}+c\,v\,\begin{bmatrix}
    \mu_0 \\ \mu_1 \\ \mu_2
\end{bmatrix} = G_2\,\Phi_0\,v.
$$
This proves the base case.

Suppose that \eqref{eq:GPhi} holds for some $k\geqslant 2$. Let
$$
\bar{v}=\begin{bmatrix}
    \nu_1 \\ \vdots \\ \nu_{k}
\end{bmatrix}\in \mathbb{R}^k.
$$
We compute
\begin{align*}
    G_{k+1}^{(1)}\begin{bmatrix}
        \bar{v} \\ 0 \\ 0
    \end{bmatrix}
=G_{k+1}^{(1)}\left(\begin{bmatrix}\nu_1 \\ \vdots \\ \nu_{k-1} \\ 0 \\ 0 \\ 0 \end{bmatrix}+\begin{bmatrix}
    0 \\ \vdots \\ 0 \\ \nu_{k} \\ 0 \\ 0
\end{bmatrix} \right).
\end{align*}
Multiplying by blocks and using the induction hypothesis, we get
\begin{align*}
 G_{k+1}^{(1)}\begin{bmatrix}
        \bar{v} \\ 0 \\ 0
    \end{bmatrix}&=  \left[\begin{array}{ccc|c}
 & & & \sigma_{k+1} \\
 & G_{k}^{(1)} & &\vdots \\
 & & &\sigma_{2k+1}\\
\hline 
\sigma_{k+1} & \ldots & \sigma_{2k+1} & \sigma_{2k+2}
\end{array}\right]\begin{bmatrix}\nu_1 \\ \vdots \\ \nu_{k-1} \\ 0 \\ 0 \\ 0 \end{bmatrix}+\nu_{k}\begin{bmatrix} \sigma_{k-1}\\ \vdots \\ \sigma_{2k-1}\\\sigma_{2k}\end{bmatrix}\\[6pt]
&= \left[\begin{array}{ccc} 
 & G_k\,\Phi_{k-2} & \\[-6pt]
\\
\hline
\sigma_{k+1} & \cdots & \sigma_{2k-1}
\end{array}\right]\begin{bmatrix}\nu_1 \\ \vdots \\ \nu_{k-1} \end{bmatrix}+\nu_{k}\begin{bmatrix} \sigma_{k-1}\\ \vdots\\ \sigma_{2k-1} \\ \sigma_{2k}\end{bmatrix}\\[3pt]
&=\left[\begin{array}{ccc|c}
 & & & \sigma_{k-1} \\
 & G_k\,\Phi_{k-2} & &\vdots \\
 & & &\sigma_{2k-1}\\
\hline 
\sigma_{k-1} & \ldots & \sigma_{2k-1} & \sigma_{2k}
\end{array}\right]\,\bar{v}.
\end{align*} 
Then, it is straightforward to verify that
$$
\left[\begin{array}{ccc|c}
 & & & \sigma_{k-1} \\
 & G_k\,\Phi_{k-2} & &\vdots \\
 & & &\sigma_{2k-1}\\
\hline 
\sigma_{k-1} & \ldots & \sigma_{2k-1} & \sigma_{2k}
\end{array}\right]= \left[\begin{array}{ccc|c}
 & & & \mu_{k+1} \\
 & G_{k} & &\vdots \\
 & & &\mu_{2k+1}\\
\hline 
\mu_{k+1} & \ldots & \mu_{2k+1} & \mu_{2k+2}
\end{array}\right]\left[\begin{array}{c|c} 
 & \mathtt{0} \\
\Phi_{k-2} & c \\
 & b \\
\hline 
\mathtt{0} & a
\end{array}\right].
$$
This proves that
$$
    G_{k+1}^{(1)}\,\begin{bmatrix}
        \bar{v} \\ 0 \\ 0
    \end{bmatrix}\,=\,G_{k+1}\,\Phi_{k-1}\,\bar{v},
$$
and, thus, \eqref{eq:GPhi} holds for $n\geqslant 2$.
\end{proof}

We have the following characterization of classical sequences of real numbers.

\begin{theorem}\label{th:firststructurerelation}
Let $\{\mu_n\}_{n\geqslant 0}$ be a sequence of real numbers such that $\det G_n\ne 0$ for $n\geqslant 0$. Let $S_{n}^{-1}\,H_{n}\,S_{n}^{-\top}$ be the Cholesky factorization of $G_{n}$ and let $\bar{s}_{n,0},\bar{s}_{n,1},\ldots,\bar{s}_{n,n}$ denote the columns of $S_{n}^{\top}$. Then $\{\mu_n\}_{n\geqslant 0}$ is classical if and only if there are real numbers $a,b,c$ satisfying 
$$
|a|+|b|+|c|>0,
$$
and real numbers $a_j,b_j,c_j$, $j\geqslant 0$, with $c_j\ne 0$, such that the set $\{ \bar{s}_{n,0}^{(1)}, \ldots, \bar{s}_{n,n}^{(1)}\}$ of vectors in $\mathbb{R}^{n+1}$ with
$$
\bar{s}_{n,j}^{(1)}=\frac{1}{j+1}\,N_{n+1}\,\bar{s}_{n+1,j+1}, \quad 0\leqslant j \leqslant n,
$$
satisfy
\begin{equation}\label{eq:firststrel}
\Phi_{n-2}\,\bar{s}_{n-2,j}^{(1)} = a_j\,\bar{s}_{n,j+2}+b_j\,\bar{s}_{n,j+1}+c_j\,\bar{s}_{n,j}, \quad 0\leqslant j \leqslant n-2,
\end{equation}
with $\Phi_n$ as defined in \eqref{def:matrizPhi}
\end{theorem}

\begin{proof}
We will use the fact that 
$$
S_{n,1}^{\top}=\left[\bar{s}_{n,0}^{(1)}\ \bar{s}_{n,1}^{(1)}\, \ldots\, \bar{s}_{n,n}^{(1)} \right],
$$
is a unit upper triangular matrix.

    Suppose that $\{\mu_n\}_{n\geqslant 0}$ is a classical sequence satisfying \eqref{eq:ttr-moments}. Since $\Phi_{n-2}\,\bar{s}_{n-2,j}^{(1)}\in \mathbb{R}^{n+1}$  and $\{\bar{s}_{n,0},\bar{s}_{n,1},\ldots,\bar{s}_{n,n}\}$ constitutes an orthogonal basis for $\mathbb{R}^{n+1}$, we can write
    $$
    \Phi_{n-2}\,\bar{s}_{n-2,j}^{(1)}=\sum_{k=0}^{n}a_{j,k}\,\bar{s}_{n,k},
    $$
    with
    $$
    a_{j,k}\,h_k\,=\,\bar{s}_{n,k}^{\top}\,G_n\,\Phi_{n-2}\,\bar{s}_{n-2,j}^{(1)}, \quad 0\leqslant k \leqslant n.
    $$
    The absence of a subindex indicating some dependence on $n$ is justified by Theorem \ref{th:choleskySn+1} and  Theorem \ref{th:choleskysigma}, which imply that the expression for $\bar{s}_{n,k}$ and $\bar{s}_{n-2,j}^{(1)}$ are independent of $n$.
    
    By \eqref{eq:GPhi} and since 
    $$
    S_{n,1}^{\top}=\left[\begin{array}{c|c} S_{n-2,1}^{\top} & \ast \\
    \hline
    \mathtt{0} & \begin{matrix} 1 & \ast \\ 0 & 1\end{matrix} \end{array} \right],
    $$
    we have
    $$
    a_{j,k}\,h_k\,=\,\bar{s}_{n,k}^{\top}\,G_n^{(1)}\,\begin{bmatrix}\bar{s}_{n-2,j}^{(1)}\\ 0 \\ 0 \end{bmatrix}=\bar{s}_{n,k}^{\top}\,G_n^{(1)}\,\bar{s}_{n,j}^{(1)}, \quad 0\leqslant k \leqslant n, \quad 0\leqslant j \leqslant n-2.
    $$
    On one hand, by Corollary \ref{th:higher-Hahn-type}, the set $\{ \bar{s}_{n,0}^{(1)}, \ldots, \bar{s}_{n,n}^{(1)}\}$ is an orthogonal basis for $\mathbb{R}^{n+1}$, and since $\bar{s}_{n,k}=\bar{s}_{n,k}^{(1)}+\alpha_{k,k-1}\,\bar{s}^{(1)}_{n,k-1}+\cdots+\alpha_{k,0}\,\bar{s}_{n,0}^{(1)}$, we get that $a_{j,k}=0$ for $0\leqslant k \leqslant j-1$, and
    $$
    a_{j,j}\,h_j=\bar{s}_{n,j}^{\top}\,G_n^{(1)}\,\bar{s}_{n,j}^{(1)}= (\bar{s}_{n,j}^{(1)})^{\top}\,G_n^{(1)}\,\bar{s}_{n,j}^{(1)}=h_j^{(1)}.
    $$
    Therefore, 
    $$
    a_{j,j}=\frac{h_j^{(1)}}{h_j}=-\frac{\lambda_{j+1}}{(j+1)^2}\ne 0.
    $$
    On the other hand, since $S_{n-2,1}^{\top}$ is upper triangular, we have
    $$
    \Phi_{n-2}\,\bar{s}_{n-2,j}^{(1)}=\begin{bmatrix}
        \ast \\ a \\ \mathtt{0}
    \end{bmatrix}\in \mathbb{R}^{n+1}, \quad 0 \leqslant j \leqslant n-2,
    $$
    where the last $n-j-2$ entries are zero. This implies that $a_{j,k}=0$ for $k\geqslant j+3$. Therefore, \eqref{eq:firststrel} holds with $a_j=a_{j,j+2}$, $b_j=a_{j,j+1}$ and $c_j=a_{j,j} \neq 0$.

    Conversely, suppose that \eqref{eq:firststrel} holds with the entries of $\Phi_{n-2}$ are real numbers $a,b,c$ satisfying $|a|+|b|+|c|>0$. On one hand, multiplying both sides of \eqref{eq:firststrel} by $\bar{s}_{n,0}^{\top}\,G_n$, we obtain for $0\leqslant j \leqslant n-2$,
    $$
    \bar{s}_{n,0}^{\top}\,G_n\left( a_j\,\bar{s}_{n,j+2}+b_j\,\bar{s}_{n,j+1}+c_j\,\bar{s}_{n,j}\right) = \bar{s}_{n,0}^{\top}\,G_n\Phi_{n-2}\,\bar{s}_{n-2,j}^{(1)}.
    $$
    By the orthogonality of $\{\bar{s}_{n,0},\bar{s}_{n,1},\ldots,\bar{s}_{n,n}\}$, we have
    $$
    \bar{s}_{n,0}^{\top}\,G_n\Phi_{n-2}\, N_{n-1}\,\bar{s}_{n-1,j+1} \,=\, (j+1) \,c_0\,h_0\,\delta_{j,0}.
    $$
    On the other hand, we have
    $$
    \bar{s}_{n-1,1}^{\top}\,G_{n-1}\,\bar{s}_{n-1,j+1}\,=\,h_1\,\delta_{j,0}=(j+1)\,h_1\,\delta_{j,0}, \quad 0 \leqslant j \leqslant n-2.
    $$
    Therefore,
    $$
    \bar{s}_{n,0}^{\top}\,G_n\Phi_{n-2}\,N_{n-1}\,\bar{s}_{n-1,j+1}\,=\,c_0\,\frac{h_0}{h_1}\,\bar{s}_{n-1,1}^{\top}\,G_{n-1}\,\bar{s}_{n-1,j+1},
    $$
    or, equivalently,
    $$
    \left(\bar{s}_{n,0}^{\top}\,G_n\,\Phi_{n-2}\,N_{n-1}-c_0\,\frac{h_0}{h_1}\,\bar{s}_{n-1,1}^{\top}\,G_{n-1} \right)\,\bar{s}_{n-1,j+1}=0, \quad 0\leqslant j \leqslant n-2.
    $$
    Since $N_{n-1}\bar{s}_{n-1,0}=\mathtt{0}$ and $\bar{s}_{n-1,1}^{\top}\,G_{n-1}\,\bar{s}_{n-1,0}=0$, we have
    $$
    \left(\bar{s}_{n,0}^{\top}\,G_n\,\Phi_{n-2}\,N_{n-1}-c_0\,\frac{h_0}{h_1}\,\bar{s}_{n-1,1}^{\top}\,G_{n-1} \right)\,S_{n-1}^{\top}=\mathtt{0}, 
    $$
    and, therefore,
     \begin{equation}\label{eq:struct}
    \bar{s}_{n,0}^{\top}\,G_n\,\Phi_{n-2}\,N_{n-1}-c_0\,\frac{h_0}{h_1}\,\bar{s}_{n-1,1}^{\top}\,G_{n-1} =\mathtt{0},
    \end{equation}
    where we have used the fact that $S_{n-1}^{\top}$ is an invertible matrix. If
$$
\bar{s}_{n-1,1}^{\top}=[s_{1,0} \ 1 \ 0 \ldots 0],
$$
then let $d=-c_0\frac{h_0}{h_1}$ and $e=d\,s_{1,0}$. In this way, the entries of \eqref{eq:struct} read
\begin{align*}
&d\,\mu_1+e\,\mu_0 = 0,\\
&(k\,a+d)\,\mu_{k+1}+(k\,b+e)\,\mu_k+k\,c\,\mu_{k-1}=0, \quad 1\leqslant k \leqslant n-1,
\end{align*}
which proves that $\{\mu_n\}_{n\geqslant 0}$ satisfies \eqref{eq:ttr-moments} and, thus, $\{\mu_n\}_{n\geqslant 0}$ is classical.
\end{proof}

We briefly recast Theorem \ref{th:firststructurerelation} in terms of polynomials. Let $\{\mu_n\}_{n\geqslant 0}$ be a classical sequence of real numbers, and let $\mathbf{u}$ be the moment functional defined as $\mu_n=\langle \mathbf{u}, x^n\rangle$, $n\geqslant 0$. Let $\{P_n(x)\}_{n\geqslant 0}$  and $\{Q_n(x)\}_{n\geqslant 0}$ let be sequences of polynomials with 
$$
P_n(x)\,=\,\bar{s}_{n,n}^{\top}\,\mathtt{X}_n \quad \text{and} \quad Q_{n}(x)\,=\,(\bar{s}_{n,n}^{(1)})^{\top}\,\mathtt{X}_n \quad n\geqslant 0.
$$
Then $\{P_n(x)\}_{n\geqslant 0}$  is an OPS associated with $\mathbf{u}$, and
$$
Q_j(x)\,=\,\frac{P_{j+1}'(x)}{j+1}, \quad 0\leqslant j \leqslant n.
$$
For $a,b,c\in \mathbb{R}$ such that $|a|+|b|+|c|>0$, and $n\geqslant 0$, the matrix $\Phi_n$ is the matrix representation of the linear mapping from $\Pi_n$ to $\Pi_{n+2}$ defined by $p(x) \mapsto \phi(x)\,p(x)$ with $\phi(x)=a\,x^2+b\,x+c$. Therefore,
$$
(\Phi_{n}\,\bar{s}_{n,n}^{(1)})^{\top}\,\mathtt{X}_{n+2}=\phi(x)\,Q_{n}(x).
$$
In this way, by Theorem \ref{th:firststructurerelation}, $\mathbf{u}$ is a classical moment functional if and only if there is a non zero polynomial $\phi(x)$ with $\deg \phi(x) \leqslant 2$, and real numbers $a_n,b_n,c_n$, $n\geqslant 0$, with $c_n\ne 0$,
$$
\phi(x)\,Q_n(x)=a_n\,p_{n+2}(x)+b_n\,p_{n+2}(x)+c_n\,p_n(x), \quad n\geqslant 0.
$$

\subsection{Second structure relation} 
The following characterization is similar to \eqref{eq:firststrel} but it has a dual flavor in the sense that the roles of $\{\bar{s}_{n,j}\}_{j=0}^n$ and $\{\bar{s}_{n,j}^{(1)}\}_{j=0}^n$ are interchanged. 

\begin{theorem}\label{th;secondstructurerelation}
Let $\{\mu_n\}_{n\geqslant 0}$ be a sequence of real numbers such that $\det G_n\ne 0$ for $n\geqslant 0$. Let $S_{n}^{-1}\,H_{n}\,S_{n}^{-\top}$ be the Cholesky factorization of $G_{n}$ and let $\bar{s}_{n,0},\bar{s}_{n,1},\ldots,\bar{s}_{n,n}$ denote the columns of $S_{n}^{\top}$. Then $\{\mu_n\}_{n\geqslant 0}$ is classical if and only if there are real numbers $\kappa_j,\xi_j$, $j\geqslant 0$, such that
\begin{equation}\label{eq:secondstr}
\bar{s}_{n,j}\,=\,\bar{s}_{n,j}^{(1)}+\kappa_j\,\bar{s}_{n,j-1}^{(1)}+\xi_j\,\bar{s}_{n,j-2}^{(1)},  \quad 0\leqslant j \leqslant n,
\end{equation}
where, by convention, we set $\bar{s}_{n,-2}^{(1)}\,=\,\bar{s}_{n,-1}^{(1)}\,=\,\mathtt{0}$.
\end{theorem}
\begin{proof}
We will use the fact that 
$$
S_{n,1}^{\top}=\left[\bar{s}_{n,0}^{(1)}\ \bar{s}_{n,1}^{(1)}\, \ldots\, \bar{s}_{n,n}^{(1)} \right],
$$
is a unit upper triangular matrix.

    Suppose that $\{\mu_n\}_{n\geqslant 0}$ is a classical sequence. By Corollary \ref{th:higher-Hahn-type}, $\{ \bar{s}_{n,0}^{(1)}, \ldots, \bar{s}_{n,n}^{(1)}\}$ is an orthogonal basis for $\mathbb{R}^{n+1}$. Since $S_{n,1}^{\top}$ is a unit upper triangular matrix, we can write
    $$
    \bar{s}_{n,j}\,=\,\bar{s}_{n,j}^{(1)}+\sum_{k=0}^{j-1} a_{j,k}\,\bar{s}_{n,k}^{(1)}, \quad 0 \leqslant j \leqslant n,
    $$
    where
    $$
    a_{j,k}\,h_k^{(1)}\,=\,(\bar{s}_{n,k}^{(1)})^{\top}\,G_n^{(1)}\,\bar{s}_{n,j}, \quad 0\leqslant k \leqslant j-1.
    $$
    As before, the absence of a subindex indicating some dependence on $n$ is justified by Theorem \ref{th:choleskySn+1} and  Theorem \ref{th:choleskysigma}, which imply that the expression for $\bar{s}_{n,k}$ and $\bar{s}_{n-2,j}^{(1)}$ are independent of $n$. Using \eqref{eq:GPhi} and Theorem \ref{th:firststructurerelation}, we get
    $$
    a_{j,k}\,h_k^{(1)}\,=\,(a_k\,\bar{s}_{n,k+2}+b_k\,\bar{s}_{n,k+1}+c_k\,\bar{s}_{n,k})^{\top}\,G_n\,\bar{s}_{n,j}, \quad 0\leqslant k \leqslant n-2.
    $$
    From the orthogonality of $\{ \bar{s}_{n,0}, \ldots, \bar{s}_{n,n}\}$, we obtain $a_{j,k}\,=\,0$ for $k\leqslant j-3$. Therefore, \eqref{eq:secondstr} holds with $\kappa_j=a_{j,j-1}$ and $\xi_j=a_{j,j-2}$.

    Conversely, suppose that there are real numbers $\kappa_j, \xi_j$, $j\geqslant 0$, such that \eqref{eq:secondstr} holds. For $n\geqslant 0$, define the $(n+1)\times (n+1)$ unit upper triangular matrices
    $$
    U_0=1, \quad U_1=\begin{bmatrix}
        1 & \kappa_1 \\
        0 & 1
    \end{bmatrix}, \quad U_n=\left[ \begin{array}{c|c}
     & \mathtt{0} \\
     U_{n-1} & \xi_n \\
      & \kappa_n \\
      \hline
      \mathtt{0} & 1
    \end{array}\right], \quad n\geqslant 2.
    $$
    Then \eqref{eq:secondstr} can be written as
    $$
    S_n^{\top}\,=\,S_{n,1}^{\top}\,U_n.
    $$
    Using $H_n=S_n\,G_n\,S_n^{\top}$, we get
    $$
    U_n\,=\,U_n\,H_n^{-1}\,S_n\,G_n\,S_n^{\top}\,=\,U_n\,H_n^{-1}\,S_n\,G_n\,S_{n,1}^{\top}\,U_n,
    $$
    which implies
    $$
    I_{n+1}\,=\,U_n\,H_n^{-1}\,S_n\,G_n\,S_{n,1}^{\top}.
    $$
    From this equality we get
    \begin{equation}\label{eq:a}
    S_{n,1}^{-\top}\,=\,U_n\,H_n^{-1}\,S_n\,G_n.
    \end{equation}
    On one hand, observe that
    \begin{equation}\label{eq:b}
    \begin{aligned}
    N_{n+1}\,S_{n+1}^{\top}&=\,\begin{bmatrix}
        \mathtt{0} & N_{n+1}\,\bar{s}_{n+1,1} & N_{n+1}\,\bar{s}_{n+1,2} & \ldots & N_{n+1}\,\bar{s}_{n+1,n+1}
    \end{bmatrix}\\
    &=\,\begin{bmatrix}
        \mathtt{0} & \bar{s}_{n,0}^{(1)} & 2\,\bar{s}_{n,1}^{(1)} & \ldots & (n+1)\,\bar{s}_{n,n}^{(1)}
    \end{bmatrix}\\
    &=\,S_{n,1}^{\top}\,N_{n+1},
    \end{aligned}
    \end{equation}
    Combining \eqref{eq:a} and \eqref{eq:b}, we get
    $$
    N_{n+1}\,=\,U_n\,H_n^{-1}\,S_n\,G_n\,N_{n+1}\,S_{n+1}^{\top}.
    $$
    On the other hand, we have
    $$
    \frac{1}{h_1}\,\bar{s}_{n+1,1}^{\top}\,G_{n+1}\,S_{n+1}^{\top}\,=\,\bar{s}_{n,0}^{\top}\,N_{n+1}.
    $$
    Therefore,
    \begin{equation}\label{eq:c}
    \frac{1}{h_1}\,\bar{s}_{n+1,1}^{\top}\,G_{n+1}\,=\,\bar{s}_{n,0}^{\top}\,U_n\,H_n^{-1}\,S_n\,G_n\,N_{n+1}.
    \end{equation}
    If we let
    \begin{equation}\label{eq:d}
    \begin{bmatrix}
        c \\ b \\ a \\ \mathtt{0}
    \end{bmatrix} = \frac{\xi_2}{h_2}\,\bar{s}_{n,2}+\frac{\kappa_1}{h_1}\,\bar{s}_{n,1}+\frac{1}{h_0}\,\bar{s}_{n,0}, \quad \text{and} \quad \begin{bmatrix}
        e \\ d \\ \mathtt{0}
    \end{bmatrix}=-\frac{1}{h_1}\,\bar{s}_{n+1,1},
    \end{equation}
    then the entries of \eqref{eq:c} read
    \begin{align*}
-e\,\mu_0-d\,\mu_{1}&= 0,\\
-e\,\mu_k-d\,\mu_{k+1}&=k\,c\,\mu_{k-1}+k\,b\,\mu_k+k\,a\,\mu_{k+1}, \quad 1\leqslant k \leqslant n+1,
\end{align*}
which proves that $\{\mu_n\}_{n\geqslant 0}$ satisfies \eqref{eq:ttr-moments} and, thus, $\{\mu_n\}_{n\geqslant 0}$ is classical.
\end{proof}

For a classical sequence $\{\mu_n\}_{n\geqslant 0}$, let $\mathbf{u}$ be the moment functional defined as $\mu_n=\langle \mathbf{u}, x^n\rangle$, $n\geqslant 0$, and let $\{P_n(x)\}_{n\geqslant 0}$  and $\{Q_n(x)\}_{n\geqslant 0}$ let be sequences of polynomials with 
$$
P_n(x)\,=\,\bar{s}_{n,n}^{\top}\,\mathtt{X}_n \quad \text{and} \quad Q_{n}(x)\,=\,(\bar{s}_{n,n}^{(1)})^{\top}\,\mathtt{X}_n \quad n\geqslant 0.
$$
Theorem \ref{th;secondstructurerelation} implies that there are real numbers $\kappa_n, \xi_n$, $n\geqslant 0$, such that 
$$
P_n(x)\,=\,Q_n(x)+\kappa_n\,Q_{n-1}(x)+\xi_n\,Q_{n-2}(x), \quad n\geqslant 0,
$$
where, by convention, $Q_{-2}(x)=Q_{-1}(x)=0$. Moreover, we deduce from \eqref{eq:d} and Theorem \ref{th:pearson} that $\mathbf{u}$ satisfies $D(\phi\,\mathbf{u})=\psi\,\mathbf{u}$ with
$$
\phi(x)=\frac{\xi_2}{h_2}P_2(x)+\frac{\kappa_1}{h_1}P_1(x)+\frac{1}{h_0}P_0(x) \quad \text{and} \quad \psi(x)=-\frac{1}{h_1}P_1(x).
$$

\subsection{Rodrigues-type formula}  

Recall that classical sequences of real numbers $\{\mu_n\}_{n\geqslant 0}$ satsisfy the three-term recurrence relation \eqref{eq:ttr-moments} which can be written in matrix form as
$$
N_{n+1}^{\top}\,G_n^{(1)}\bar{s}_{n,0}+G_n\,\begin{bmatrix}
    e \\ d \\ \mathtt{0}
\end{bmatrix}\,=\,\mathtt{0}, \quad n\geqslant 1,
$$
with $G_n^{(1)}$ as defined in \eqref{def:Gn1}. The following characterization shows that classical sequences satisfy higher order recurrence relations which can be written in matrix form as well.

\begin{theorem}\label{th:rodrigues}
Let $\{\mu_n\}_{n\geqslant 0}$ be a sequence of real numbers such that $\det G_n\ne 0$, $n\geqslant 0$. Let $S_{n}^{-1}\,H_{n}\,S_{n}^{-\top}$ be the Cholesky factorization of $G_{n}$ and let $\bar{s}_{n,0},\bar{s}_{n,1},\ldots,\bar{s}_{n,n}$ denote the columns of $S_{n}^{\top}$. Then $\{\mu_n\}_{n\geqslant 0}$ is classical if and only if there are $a,b,c\in \mathbb{R}$ such that $|a|+|b|+|c|>0$, and non zero real numbers $\varpi_k$, $k\geqslant 1$, such that
\begin{equation}\label{eq:rodrigues}
    N_{n+k}^{\top}\cdots N_{n+1}^{\top}\,G_n^{(k)}\,\bar{s}_{n,0}\,=\,\varpi_k\,\,G_{n+k}\,\bar{s}_{n+k,k}, \quad n\geqslant 1,  \quad 1\leqslant k \leqslant n,
\end{equation}
with $G_n^{(k)}$ as defined in \eqref{def:Gnk}.
\end{theorem}
\begin{proof}
    Suppose that $\{\mu_n\}_{n\geqslant 0}$ is classical satisfying \eqref{eq:ttr-moments}. For $k\geqslant 0$, by Corollary \ref{th:higher-Hahn-type}, the set $\{ \bar{s}_{n,0}^{(k)}, \ldots, \bar{s}_{n,n}^{(k)}\}$ of vectors in $\mathbb{R}^{n+1}$ with
$$
\bar{s}_{n,j}^{(k)}=\frac{1}{j+1}\,N_{n+1}\,\bar{s}_{n+1,j+1}^{(k-1)}, \quad 0\leqslant j \leqslant n,
$$
where $s_{n,j}^{(0)}=\bar{s}_{n,j}$, constitutes an orthogonal basis for $\mathbb{R}^{n+1}$ with respect to the bilinear form associated with $\{\sigma_n^{(k)}\}_{n\geqslant 0}$. Let
$$
S_{n,k}^{\top}=\left[\bar{s}_{n,0}^{(k)}\ \bar{s}_{n,1}^{(k)}\, \ldots\, \bar{s}_{n,n}^{(k)} \right].
$$

On one hand, observe that for $k\geqslant 1$,
\begin{equation*}
    \begin{aligned}
    N_{n+1}\,S_{n+1,k-1}^{\top}&=\,\begin{bmatrix}
        \mathtt{0} & N_{n+1}\,\bar{s}_{n+1,1}^{(k-1)} & N_{n+1}\,\bar{s}_{n+1,2}^{(k-1)} & \ldots & N_{n+1}\,\bar{s}_{n+1,n+1}^{(k-1)}
    \end{bmatrix}\\
    &=\,\begin{bmatrix}
        \mathtt{0} & \bar{s}_{n,0}^{(k)} & 2\,\bar{s}_{n,1}^{(k)} & \ldots & (n+1)\,\bar{s}_{n,n}^{(k)}
    \end{bmatrix}\\
    &=\,S_{n,k}^{\top}\,N_{n+1},
    \end{aligned}
    \end{equation*}
    Then, since $\bar{s}^{(k)}_{n,0}=\bar{s}_{n,0}$ for $k\geqslant 1$,
    \begin{align*}
    S_{n+k}\,N_{n+k}^{\top}\cdots N_{n+1}^{\top}\,G_n^{(k)}\,\bar{s}_{n,0}&=\,N_{n+k}^{\top}\cdots N_{n+1}^{\top}\,S_{n,k}\,G_n^{(k)}\,\bar{s}_{n,0}^{(k)}\\
    &=\,k!\,h_0^{(k)}\,\bar{e}_k\in \mathbb{R}^{n+k+1},
    \end{align*}
    where $\bar{e}_k$ denotes the $k+1$-th column of the identity matrix $I_{n+k+1}$. On the other hand, 
    $$
    \frac{1}{h_k}\,S_{n+k}\,G_{n+k}\,\bar{s}_{n+k,k}\,=\,\bar{e}_k.
    $$
    Since $S_{n+k}$ is invertible, it follows that
    $$
    N_{n+k}^{\top}\cdots N_{n+1}^{\top}\,G_n^{(k)}\,\bar{s}_{n,0}\,=\,k!\,\frac{h_0^{(k)}}{h_k}\,\,G_{n+k}\,\bar{s}_{n+k,k}, \quad k\geqslant 1.
    $$
    Hence, \eqref{eq:rodrigues} holds with $\varpi_k=k!\,\frac{h_0^{(k)}}{h_k}$.

    Conversely, suppose that there are $a,b,c\in \mathbb{R}$ such that $|a|+|b|+|c|>0$, and non zero real numbers $\varpi_k$, $k\geqslant 1$, such that \eqref{eq:rodrigues} holds. In particular, for $k=1$, we have
    \begin{equation}\label{eq:e}
    N_{n+1}^{\top}\,G_n^{(1)}\,\bar{s}_{n,0}\,=\,\varpi_1\,\,G_{n+1}\,\bar{s}_{n+1,1}, \quad n\geqslant 0.
    \end{equation}
    If we let
    $$
    \begin{bmatrix}
        e \\ d 
    \end{bmatrix} =-\varpi_1\,\bar{s}_{1,1},
    $$
    and since $\bar{s}_{n+1,1}=\begin{bmatrix}\bar{s}_{1,1} \\ \mathtt{0} \end{bmatrix}$, then the entries of \eqref{eq:e} read
        \begin{align*}
-e\,\mu_0-d\,\mu_{1}&= 0,\\
-e\,\mu_k-d\,\mu_{k+1}&=k\,c\,\mu_{k-1}+k\,b\,\mu_k+k\,a\,\mu_{k+1}, \quad 1\leqslant k \leqslant n+2,
\end{align*}
which proves that $\{\mu_n\}_{n\geqslant 0}$ satisfies \eqref{eq:ttr-moments} and, thus, $\{\mu_n\}_{n\geqslant 0}$ is classical.
\end{proof}

We remark that if $\{\mu_n\}_{n\geqslant 0}$ is a classical sequence of real numbers, and $\mathbf{u}$ be the moment functional defined as $\mu_n=\langle \mathbf{u}, x^n\rangle$, $n\geqslant 0$, then the entries of \eqref{eq:e} can be written as
$$
\langle D(\phi\,\mathbf{u}), x^k \rangle \,=\,\langle \psi\,\mathbf{u}, x^k\rangle, \quad 0\leqslant k \leqslant n+2,
$$
where $\phi(x)=a\,x^2+b\,x+c$ and $\psi(x)=d\,x+e$, which holds for all $n\geqslant 0$. Hence, $\mathbf{u}$ satisfies \eqref{pearson}. Moreover, it is straightforward, but tedious, to verify that the entries of \eqref{eq:rodrigues} can be written as
$$
\langle D^k(\phi^k\,\mathbf{u}), x^j\rangle \,=\,\langle \varpi_k\,P_k\,\mathbf{u}, x^j\rangle, \quad 0\leqslant j \leqslant n+k+1,
$$
where $P_k(x)=\bar{s}_{n+k,k}^{\top}\,\mathtt{X}_{n+k}$, $k\geqslant 1$, which holds for $n\geqslant 1$. Hence, $\mathbf{u}$ satisfies $D^k(\phi^k\,\mathbf{u})=\varpi_k\,P_k\,\mathbf{u}$ for $k\geqslant 1$ (and holds for $k=0$ with $\varpi_0=1$).



\begin{thebibliography}{99}

\bibitem{AN06}
R. Álvarez-Nodarse, 
\textit{On characterizations of classical polynomials},
J. Comput. Appl. Math. \textbf{196} (2006), 320--337.

\bibitem{ASC72} W. Al Salam, T. S. Chihara, \textit{Another characterization of the classical orthogonal polynomials}, SIAM J. Math. Anal. \textbf{3}(1), (1972), 65-70.

\bibitem{B29}
S. Bochner,  
\textit{Über Sturm-Liouvillesche Polynomsysteme},  
Math. Z. \textbf{29} (1929), no. 1, 730--736.

\bibitem{CGN20}
F. A. Costabile, M. I. Gualtieri, A. Napoli, 
\textit{Matrix Calculus-Based Approach to Orthogonal Polynomial Sequences},
Mediterr. J. Math. \textbf{17} (118) (2020).

\bibitem{D20}
D. Dominici,
\textit{Matrix factorizations and orthogonal polynomials},
Random Matrices Theor. \textbf{9} (2020), 2040003, 33pp.

\bibitem{Ger40}
J. Geronimus,
\textit{On polynomials orthogonal with regard to a given sequence of numbers},
Comm. Inst. Math. Mec. Univ. Kharkoff [Zapiski Inst. Math. Mech.] \textbf{17} (4) (1940), 3--18 (Russian).

\bibitem{GMM21} J. C. Garcia-Ardila, F. Marcellán, M. E. Marriaga, \textit{Orthogonal Polynomials and Linear Functionals. An Algebraic Approach and Applications}, European Mathematical Society (EMS), Zürich, 2021.  
\bibitem{GM21} 
J. C. Garcia-Ardila, M. E. Marriaga, 
\textit{On Sobolev bilinear forms and polynomial solutions of second-order differential equations},  
Rev. Real Acad. Cienc. Exactas Fis. Nat. Ser. A-Mat.  (2021).

\bibitem{GV13} 
G. H. Golub, C. F. Van Loan,
\textit{Matrix computations},
Johns Hopkins Studies in Mathematical Sciences, Johns Hopkins University Press, Baltimore, MD, 4th ed., 2013.

\bibitem{Hahn35}
W. Hahn,  
\textit{\"{U}ber die {J}acobischen Polynome und Zwei Verwandte Polynomklassen}, 
Math. Z. 39 (1935), 634--638.
	
\bibitem{Hahn36} 
W. Hahn. 
\textit{\"Uber Differentialgleichunqen f\"{u}r Orthoqonalpolynome.}
Monat. Math. {\bf 95}, (1983), p. 269--274.


\bibitem{HJ85}
R. A. Horn, C. R. Johnson,
\textit{Matrix Analysis},
Second ed., Cambridge University Presss, Cambridge, 2013.

\bibitem{Krall41}
H. L. Krall, 
\textit{On derivatives of orthogonal polynomials II}, 
Bull. Amer. Math. Soc. \textbf{47} (1941), 261--264.

\bibitem{Krall36}
H. L. Krall,
\textit{On higher derivatives orthogonal polynomials},
Bull. Amer. Math. Soc. \textbf{42} (1936), 423--428.

\bibitem{M2021}
M. Mañas, 
\textit{Revisiting Biorthogonal Polynomials: An LU Factorization Discussion},
In: Marcellán, F., Huertas, E.J. (eds) Orthogonal Polynomials: Current Trends and Applications. SEMA SIMAI Springer Series, vol 22. Springer, Cham.


\bibitem{MBP94} F.~Marcell\'an, A.~Branquinho, J.~Petronilho, \textit{Classical orthogonal polynomials: a functional approach}, Acta Appl. Math. \textbf{34} (1994), 283-303.

\bibitem{MP94} 
F. Marcellán, J. Petronilho, \textit{On the solution of some distributional differential equations: existence and characterizations of the classical moment functionals}, Integral Transform. Spec. Funct. \textbf{2} (3) (1994), 185--218.

\bibitem{Ma87} 
P. Maroni, 
\textit{Prol\'{e}gom\`enes \`{a} l'\'etude des polyn\^omes semiclassiques}, 
Ann. Mat. Pura Appl. (4) \textbf{149} (1987), 165--184.


\bibitem{PR87}
M. Piñar, V. Ramírez,
\textit{Matrix interpretation of formal orthogonal polynomials for non-definite functionals}, 
J. Comput. Appl. Math. \textbf{18} (1987) 265--277.


\bibitem{Sz78}
G. Szeg\H{o},
\textit{Orthogonal polynomials},
4th edition, vol.  23. Amer. Math. Soc. Colloq. Publ., Amer. Math. Soc., Providence RI, 1975.

\bibitem{Tri70} 
F. Tricomi, 
\textit{Vorlesungen \"uber orthogonalreihen}, 
2nd edn, Springer-Verlag, Berlin, 1970.


\bibitem{V13}
L. Verde-Star,
\textit{Characterization and construction of classical orthogonal polynomials using a matrix approach},
Linear Algebra Appl. \textbf{438} (2013), 3635--3648.

\bibitem{V2021}
L. Verde-Star,
\textit{Infinite Matrices in the Theory of Orthogonal Polynomials},
In: Marcellán, F., Huertas, E.J. (eds) Orthogonal Polynomials: Current Trends and Applications. SEMA SIMAI Springer Series, vol 22. Springer, Cham.

\end{thebibliography}
\end{document}